\documentclass[10pt,reqno]{amsart}

\usepackage{amsmath, amsfonts,epsfig, amssymb,amsbsy,eucal,mathrsfs,float, amsthm,mathtools,comment}
\usepackage[usenames,dvipsnames]{xcolor} 
\usepackage[normalem]{ulem}
\usepackage{textpos}
\usepackage{graphicx}
\usepackage{dynkin-diagrams}
\usepackage{fullpage}
\usepackage{tikz}
\usepackage{enumitem}
\setlist[itemize]{leftmargin=1.8cm}
\usepackage[all]{xy}
\usepackage{tikz-cd}
\usepackage{tikz}
\usepackage{hyperref}
\hypersetup{
colorlinks = true,
urlcolor = blue,
citecolor = .
}
\usepackage{todonotes}
\numberwithin{equation}{section}
\theoremstyle{plain}
\newtheorem{maintheorem}{Theorem}
\newtheorem{theorem}{Theorem}[section]
\newtheorem{lemma}[theorem]{Lemma}
\newtheorem{proposition}[theorem]{Proposition}
\newtheorem{corollary}[theorem]{Corollary}

\theoremstyle{definition}

\newtheorem*{ack}{Acknowledgements}
\newtheorem{remark}[theorem]{Remark}
\theoremstyle{remark}

\newcommand{\OG}{OG}
\DeclareMathOperator{\Ext}{Ext}
\DeclareMathOperator{\Hom}{Hom}

\DeclareMathOperator{\Spin}{Spin}
\DeclareMathOperator{\Pic}{Pic}
\DeclareMathOperator{\rk}{rk}

\DeclareMathOperator{\Imm}{Im}
\DeclareMathOperator{\id}{id}
\renewcommand{\ss}{\mathbf{ss}}
\newcommand{\cHom}{\mathcal{H}om}
\newcommand{\is}{{i_s}}
\DeclareMathOperator{\SL}{SL}

\newcommand{\bC}{{\mathbb C}}

\newcommand{\bP}{{\mathbb P}}

\newcommand{\rE}{\mathrm{E}}
\newcommand{\rG}{\mathrm{G}}
\newcommand{\rP}{\mathrm{P}}

\newcommand{\rV}{\mathrm{V}}
\newcommand{\cA}{\mathcal{A}}
\newcommand{\cB}{\mathcal{B}}

\newcommand{\cD}{\mathcal{D}}
\newcommand{\cE}{\mathcal{E}}
\newcommand{\cF}{\mathcal{F}}

\newcommand{\cN}{\mathcal{N}}
\newcommand{\cO}{\mathcal{O}}

\newcommand{\cU}{\mathcal{U}}
\newcommand{\cL}{\mathcal{L}}

\newcommand{\rL}{\mathrm{L}}
\newcommand{\HH}{H}
\newcommand{\Db}{{\mathbf D^{\mathrm{b}}}}

\title{Derived category of the spinor 15-fold}

\author[V.~Benedetti]{Vladimiro Benedetti}
\address{Universit\'e Côte d’Azur, CNRS – Laboratoire J.-A. Dieudonné, Parc Valrose, F-06108 Nice CEDEX 2, France}
\email{vladimiro.benedetti@univ-cotedazur.fr}

\author[D.~Faenzi]{Daniele Faenzi}
\address{Institut de Mathématiques de Bourgogne,
UMR CNRS 5584,
Université de Bourgogne et Franche-Comté,
9 Avenue Alain Savary,
BP 47870,
21078 Dijon Cedex,
France}
\email{daniele.faenzi@u-bourgogne.fr}

\author[M.~Smirnov]{Maxim Smirnov}
\address{Institut für Mathematik, Fakultäten Mathematisch-Naturwissenschaftlich-Technische Fakultät, Universitätsstrasse 14, 86159 Augsburg, Germany}
\email{maxim.smirnov@math.uni-augsburg.de}
\email{maxim.n.smirnov@gmail.com}



\begin{document}

\begin{abstract}
We construct a full exceptional Lefschetz collection on the spinor 15-fold consisting of a connected component of the space of orthogonal 6-dimensional subspaces of a 12-dimensional complex vector space, isotropic with respect of a fixed non-degenerate quadratic form.
The collection is made of 2 twists of a 4-item block and 8 twists of a 3-item block, confirming a conjecture of Kuznetsov and Smirnov. We speculate that a similar collection might work for the Freudenthal $\rE_7$-variety.
\end{abstract}

\maketitle

\sloppy 

\section{Introduction}

It is widely expected that, for any parabolic subgroup $\rP$ of a reductive complex algebraic group $\rG$, the associated rational homogeneous variety $X=\rG/\rP$ admits a full exceptional $\rG$-equivariant collection. Moreover the objects of such a collection should admit a natural partial order induced by the Bruhat-Chevalley order, see for instance \cite{bohning:homogeneous} for an account. While full exceptional collections were given for flags of type $A_n$ and quadrics in \cite{beilinson:linear, kapranov:homogeneous}, in the remaining classical types exceptional collections of maximal length were constructed much later, see \cite{kuznetsov-polishchuk}. 
Some more cases admitting full exceptional collections were studied, notably for isotropic grassmannians in the symplectic case, we refer for instance to \cite{samokhin:homogeneous, polishchuk-samokhin, fonarev:lagrangian}. Full exceptional collections on some homogeneous varieties of exceptional type were studied in \cite{faenzi-manivel, belmans-kuznetsov-smirnov}.

A slightly different point of view on the structure of the derived category and on exceptional collections stems from homological projective duality, as in \cite{kuznetsov:hpd}. In this context, the emphasis is on Lefschetz properties with respect to a given ample line bundle $\cO_X(1)$, so that a full exceptional collection should be obtained from a first set of objects by twisting them with $\cO_X(t)$, for $t=0,\ldots,\ell-1$ and occasionally removing some objects. Here $\ell$ is some integer which is often the Fano index of $X$, see below.
Full exceptional Lefschetz collections were given in some classical and exceptional types in \cite{kuznetsov:isotropic, fonarev:grassmannians, faenzi-manivel, belmans-kuznetsov-smirnov}.
The question of when one should remove objects along the construction of a Lefschetz collection is a very interesting point giving rise to the study of residual categories, conjecturally related to the structure of the quantum cohomology of $X$, according to a refinement of Dubrovin's conjecture, see \cite{kuznetsov-smirnov:residual_classical, kuznetsov-smirnov:residual_grassmannian}.

In this paper we focus on two specific varieties, one of classical type and the other of exceptional type, constructing a full exceptional Lefschetz collection on the first one and providing numerical evidence on the second one, based on the ansatz that they should share some common features as they sit on the same row of Freudenthal's magic square related to real divison algebras, cf. \cite{LANMAN}.
We write $X_m$ for the varieties sitting in the third row of the Freudenthal's magic square, where the index $m$ refers to the dimension of the corresponding real division algebra $\mathbb{A}_m$.
These varieties are homogeneous for the action of a group $\rG$ listed below. They are Fano varieties whose Picard group is generated by a very ample line bundle $\cO_{X_m}(1)$, hence $\omega_{X_m} \simeq \cO_{X_m}(-\iota_{X_m})$ for some integer $\iota_{X_m}$ called the Fano index of $X_m$.
We have $\dim(X_m)=3(m+1)$ and $\iota_{X_m}=2(m+1)$.
\begin{center}
    \begin{tabular}{c|c|c|c|c}
    $m$ & 1 & 2 & 4 & 8 \\ 
\hline \hline
        $\mathbb{A}_m$ & $\mathbb{R}$ & $\mathbb{C}$ & $\mathbb{H}$ & $\mathbb{O}$ \\
\hline
        $\rG$ & $\mathrm{Sp}_3$ & $\mathrm{GL}_6$ & $\Spin_{12}$ & $\rE_7$ \\
\hline        
        $X_m$ & $LG(3,6)$ & $G(3,6)$ & $\OG_+(6,12)$ & $\rE_7/\rP_7$   \\
\hline
        $\rk(K_0(X_m))$ &8 & 20 & 32 & 56
    \end{tabular}
\end{center}

Excluding $LG(3,6)$, that does not quite fit into this picture, we have $\rk(K_0(X_m))=6m+8$. We would expect that for $m=2,4,8$ the derived category of $X_m$ has a full exceptional Lefschetz collection of the following form:
\[
\left(\cA,\cA(1),\cB(2),\ldots,\cB(2m+1)\right)
\]
with:
\begin{equation} \label{collection with m}
  \begin{aligned}
    & \cA = \left( \cO_X, O, P, Q \right) \\
    & \cB = \left( \cO_X, O, P \right)
  \end{aligned}
\end{equation}

Here, denoting by $\cU_\omega$ the irreducible $\rG$-homogeneous bundle of maximal weight $\omega$, the bundles $O$, $P$ and $Q$ should be, respectively, $\cU_{\omega_1}$, $\wedge^2 \cU_{\omega_1}$ and $S^{2,1} \cU_{\omega_1}$ with the caveat that, in case such bundles are not exceptional, we should replace them by some equivariant extension with homogenous bundles of lower maximal weight (for precise definitions see the next section) or projections on the semiorthogonal summand we are interested in.
For $m=2$, i.e. for $G(3,6)$, no extension is necessary. The resulting full exceptional collection was studied in \cite{Deliu:PhD} in the attemt to verify Homological Projective Duality for $G(3,6)$.
On the other hand, this gets more tricky for $m=4$ and $m=8$.

The goal of this paper is to prove the statement for $m=4$ and provide a partial proof of a closely related statement for $m=8$. For $m=4$ we prove:
\begin{maintheorem}
Let $X = \OG_+(6,12)$ and set $O=\cU_{\omega_1}$. Then, there are unique $\Spin_{12}$-homogeneous exceptional bundles $P$ and $Q$ fitting into:
\[
    0 \to \cO_X \to P \to \cU_{\omega_2} \to 0, \qquad     0 \to \cU_{\omega_1} \to Q \to \cU_{\omega_1+\omega_2} \to 0,
\]
such that, defining  $\cA$ and $\cB$ as in \eqref{collection with m}, we get a full Lefschetz exceptional collection:
\[
\Db(X)=\langle \cA,\cA(1),\cB(2),\ldots,\cB(9)\rangle.
\]
Moreover, $Q'=\rL_{\langle \cB \rangle}(Q)$ is a homogeneous exceptional bundle and 
$Q$ and $Q'(1)$ are completely orthogonal.
\end{maintheorem}

This overall proves \cite{kuznetsov-smirnov:residual_classical, kuznetsov-smirnov:residual_grassmannian}, including the statement about the complete orthogonality of the generators of the residual category with respect to the rectangular part of the Lefschetz collection.

For $m=8$ and $X:=X_8=\rE_7/\rP_7$, we prove a weaker result. Let us define $O$ as the unique non-trivial $\rE_7$-equivariant extension fitting into 
$$0\to \cO_{X}\to O \to \cU_{\omega_1}\to 0.$$

Let us define $P$ as the projection of $\cU_{\omega_3}$ to the left orthogonal of 
$\langle \cO_{X}(1),O(1),\dots,\cO_{X}(18),O(18)\rangle$, and $Q$ as the projection (see Remark \ref{rem_beware}) of $\cU_{\omega_1+\omega_3}$ to the left orthogonal of $\langle \cO_{X}(1),O(1),P(1),\dots,\cO_{X}(18),O(18),P(18)\rangle$.
\begin{maintheorem}
On $X = \rE_7/\rP_7$ the collection $\left( \cO_{X},O,\dots,\cO_{X}(17),O(17)\right)$ is exceptional. Moreover, defining $\cA$ and $\cB$ as in \eqref{collection with m}, we get a numerical exceptional collection of maximal length:
\[
\left(\cA,\cA(1),\cB(2),\ldots,\cB(17)\right), \quad\text{with} \quad K_0(X)=K_0(\langle \cA,\cA(1),\cB(2),\ldots,\cB(17)\rangle).
\]
\end{maintheorem}
Here, by numerical exceptional collection we mean a collection $E_1,\dots,E_r$ whose numerical properties reproduce those of an exceptional collection: $\chi(E_i,E_j)=0$ if $i>j$ and $\chi(E_i,E_i)=1$ for all $i$. Of course having a numerical exceptional collection is a priori a much weaker condition than having an exceptional collection (not to mention having a full exceptional collection). However, due to the analogy with the other cases of the Freudenthal magic square, we believe that this collection is indeed a full exceptional collection. 

The paper is organised as follows. In Section \ref{section:collection} we define our Lefschtez collection. The main tools are the theorem of Borel-Bott-Weil and a result about non-degeneracy of cup-product owing to Dimitrov and Roth. In Section \ref{warming up} we outline our strategy to prove fullness and use to reprove fullness of a natural Lefschetz collection on $OG_+(5,10)$. Here we use a complex constructed in Section \ref{section:dissecting}, where we also construct an analogous complex for $OG_+(6,12)$ which in turn we use Sections \ref{more} and \ref{section:fullness}.
In Section \ref{more} we show that certain homogeneous bundles belong to the subcategory $\cD$ generated by our exceptional collection. We use this in Section \ref{section:fullness} to prove fullness of our collection on $OG_+(6,12)$.
In Section \ref{section:E7} we provide some remarks on our numerical exceptional collection on the Freudenthal variety $\rE_7/\rP_7$.
\begin{ack}
V.B. and D.F. partially supported by FanoHK ANR-20-CE40-0023, SupToPhAG/EIPHI ANR-17-EURE-0002, Région Bourgogne-Franche-Comté, Feder Bourgogne and Bridges ANR-21-CE40-0017.
We warmly thank Sasha Kuznetsov for useful 
discussions.
\end{ack}

\section{A Lefschetz exceptional collection on the spinor 15-fold}
\label{section:collection}

Here we begin by sketching the exceptional collection we want to work with.
We first fix some setting about spinor varieties and homogeneous bundles over them, then define the bundles appearing in the desired Lefschetz collection and finally show that this is indeed an exceptional Lefschetz collection.

\subsection{Homogeneous bundles on spinor varieties} 
We consider the group $\Spin_{2n}$, namely the universal cover of the group of linear automorphisms of $\bC^{2n}$ preserving a non-degenerate quadratic form $q$. Let $\rP_n$ be the parabolic subgroup of $\Spin_{2n}$ defining the spinor Grassmannian $X=\Spin_{2n}/\rP_n:=\OG_+(n,2n)$, one of the two isomorphic connected components parametrizing $n$ dimensional isotropic subspaces of a $2n$ dimensional subspace endowed with a non-degenerate symmetric form. Let us also denote by $L(\rP_n)$ its Levi factor. 

We will denote by $\cU_\omega$ the homogeneous bundle on $\Spin_{2n}/\rP_n$ associated to the $L(\rP_n)$-weight $\omega$. We write $\cO_X(1):=\cU_{\omega_n}$ and $\cU:=\cU_{\omega_1}^\vee$. These correspond to the ample generator of $\Pic(X)$, providing the equivariant embedding of $X$ into $\bP(V^{\omega_n})$, and to the tautological sub-bundle on $G(n,2n)$, restricted to $X$. Here we denoted by $V^\lambda$ the $\Spin_{2n}$-representation of highest weight $\lambda$.

Unless further notice, we will set $n=6$ from now on and work on $X=\Spin_{12}/\rP_6$.
This is the spinor 15-fold that we are interested in. It is a Fano variety of Picard number one and index 10. The rank of its $K_0$ group is 32.
We note that
\begin{equation} \label{basic-isos}
  \cU = \cU_{\omega_1}^\vee \simeq \cU_{\omega_5}(-1), \quad \wedge^2 \cU \simeq  \cU_{\omega_2}^\vee \simeq \cU_{\omega_4}(-2), \quad
  \Sigma^{2,1} \cU \simeq \cU_{\omega_1 + \omega_2}^\vee \simeq \cU_{\omega_4+\omega_5}(-3)
\end{equation}

We denote by $\rL_\cE(\cF)$ the left mutation of an object $\cF$ about an object $\cE$ of $\Db(X)$.

\subsection{The bundles of the exceptional collection}

Let us introduce the homogeneous vector bundles appearing in our exceptional collection.

\begin{lemma}\label{lemma:canonical-extensions}
  On $X$ we have a canonical $\Spin_{2n}$-equivariant exceptional bundle $P$ fitting into:
  \begin{equation} \label{defP}
    0 \to \cO_X \to P \to \cU_{\omega_2} \to 0.
  \end{equation}
    Moreover, $P^\vee(2)$ is the normal bundle of $X$ inside $\bP(V^{\omega_6})$, while $\cU_{\omega_2}$ is the tangent bundle of $X$.
  \end{lemma}

\begin{proof}
The tangent bundle of $X$ is well-known to be $\wedge^2\cU_{\omega_1} \simeq \cU_{\omega_2}$, while the tangent bundle of $\bP(V^{\omega_6})$ restricted to $X$ is the quotient $V^{\omega_6}\otimes \cO_X(1)/\cO_X$. Since the irreducible factors of $V^{\omega_6}\otimes \cO_X(1)$ are $\cO_X$, $\cU_{\omega_2}$, $\cU_{\omega_2}^\vee(2)$ and $\cO_X(2)$, we obtain that the normal bundle $\cN$ of $X$ inside $\bP(V^{\omega_6})$ is a $\Spin_{12}$-equivariant extension $\gamma \in \Ext_X^1(\cU_{\omega_2},\cO_X)$ giving:
\[
0 \to \cU_{\omega_2}^\vee(2) \to \cN \to \cO_X(2) \to 0
\] 

By Bott-Borel-Weil (BBW) Theorem, we have $\Ext_X^1(\cO_X(2),\cU_{\omega_2}^\vee(2))=H^1(X,\cU_{\omega_2}^\vee)\simeq \bC$. Hence the sheaf fitting as middle term of a non-trivial extension as above is unique. Since $\cN(-1)$ is a quotient of $V^{\omega_6}\otimes \cO_X$ of half its rank, by autoduality of $V^{\omega_6}$ we get an exact sequence 
$$ 0 \to \cN^\vee(1) \to V^{\omega_6}\otimes \cO_X \to \cN(-1) \to 0. $$
Since $\cO_X(-2)$ and $\cU_{\omega_2}(-2)$ have no cohomology, $\cN^\vee$ has no cohomology as well. From the short exact sequence above we deduce that $\cN$ is a non-trivial extension, and thus $\cN=P^\vee(2)$; indeed, if it were not the case, one would deduce that $\bC\simeq H^0(\cO_X)\simeq H^0(\cN(-2))\simeq V^{\omega_6}\otimes H^0(\cO_X(-1))$, which is false. 

Since 
$\cU_{\omega_2}(-2)$ has no cohomology and $\cU_{\omega_2}\otimes \cU_{\omega_2}(-2)$ has no cohomology except for 
$H^1(\cU_{\omega_2}\otimes \cU_{\omega_2}(-2))=\bC$, we get that 
$\cU_{\omega_2}\otimes P(-2)$ has no cohomology except for 
$H^1(\cU_{\omega_2}\otimes P(-2))=\bC$. By twisting the exact sequence defining $P$ by $P(-2)$ we deduce that $P\otimes P(-2)$ has no cohomology except for $H^1(P\otimes P(-2))=\bC$. Now let us consider the exact sequence
$$ 0\to P\otimes P(-2) \to V^{\omega_6}\otimes P(-1) \to P^\vee\otimes P \to 0 .$$
Since $\cO_X(-1)$ and $\cU_{\omega_2}(-1)$ have no cohomology, the same is true for $P(-1)$ and $V^{\omega_6}\otimes P(-1)$. We deduce that $H^0(P^\vee \otimes P)=\bC$ and all other cohomologies of $P^\vee \otimes P$ vanish.
\end{proof}

\begin{lemma} \label{lemma:canonical-extensionQ}
  On $X$, we have a $\Spin_{12}$-homogeneous exceptional bundle $Q$ fitting into a canonical equivariant extension:
  \begin{equation} \label{defQ}
    0 \to \cU_{\omega_1} \to Q \to \cU_{\omega_1 + \omega_2} \to 0
  \end{equation}
    Moreover, we have $\Ext_X^\bullet(Q,Q(-1))=0$.
\end{lemma}

\begin{proof}
We recall \eqref{basic-isos} and use:
\begin{equation}
  \label{45*1}
  \cU_{\omega_4+\omega_5} \otimes \cU_{\omega_1} \simeq
  \cU_{2\omega_5}(1) \oplus   \cU_{\omega_4}(1) \oplus   \cU_{\omega_1+\omega_4+\omega_5}.
\end{equation}
 
We compute $H^\bullet(\cU_{2\omega_5}(-2))=H^\bullet(\cU_{\omega_1+\omega_4+\omega_5}(-3))=0$, hence:
\begin{equation}
  \label{ext1V}
  \Ext_X^\bullet(\cU_{\omega_1+\omega_2},\cU_{\omega_1}) = 
  \Ext_X^1(\cU_{\omega_1+\omega_2},\cU_{\omega_1}) 
  \simeq  H^1(\cU_{\omega_4+\omega_5} \otimes \cU_{\omega_1}(-3)) \simeq H^1(\cU_{\omega_4}(-2))= \bC.
\end{equation}

Choosing a nonzero element $\zeta$ of
$\Ext_X^1(\cU_{\omega_1+\omega_2},\cU_{\omega_1}) \simeq \bC$  defines the desired  equivariant vector bundle $Q$.

\medskip

To compute $\Ext_X^\bullet(Q, Q)$, we consider:
  \begin{align}
    \nonumber \cU_{\omega_1+\omega_2}^\vee \otimes \cU_{\omega_1} & \simeq  \cU_{\omega_1}\otimes \cU_{\omega_4+\omega_5} (-3), \\
    \label{the second} \cU_{\omega_1}^\vee \otimes \cU_{\omega_1+\omega_2} & \simeq \cU_{\omega_1+\omega_2}\otimes \cU_{\omega_5}(-1),\\
    \label{the third}   \cU_{\omega_1+\omega_2}^\vee \otimes \cU_{\omega_1+\omega_2} & \simeq
    \cU_{\omega_1+\omega_2} \otimes \cU_{\omega_4+\omega_5}(-3).
  \end{align}
  We computed the first item and its
  cohomology in \eqref{45*1} and \eqref{ext1V}.
  Using this, the fact that $\cU_{\omega_1}$ is exceptional and that $Q$ is defined by the non-zero extension $\zeta$, applying $\Ext_X^\bullet(-,\cU_{\omega_1})$ to the sequence \eqref{defQ} defining $Q$ we get
  \begin{equation} \label{Q vs U*}
      \Ext_X^\bullet(Q,\cU_{\omega_1})=0.
  \end{equation} Therefore:
  \[
    \Ext_X^\bullet(Q, Q) \simeq H^\bullet(Q^\vee \otimes \cU_{\omega_1+\omega_2}).
  \]

  To compute the term on the right-hand-side, we need to compute the
  cohomology of \eqref{the second} and \eqref{the third}.
  For \eqref{the second} we get:
  \begin{align}
    \label{125 I} \cU_{\omega_1+\omega_2} \otimes  \cU_{\omega_5}
    \simeq &\, \cU_{\omega_2}(1)  && H^\bullet(\cU_{\omega_2}) \simeq
    H^0(\cU_{\omega_2}) \simeq V^{\omega_2} , \\
    \label{125 II}\oplus &\, \cU_{2\omega_1}(1) && H^\bullet(\cU_{2\omega_1}) \simeq
    H^0(\cU_{2\omega_1}) \simeq V^{2\omega_1},\\
    \label{125 III}\oplus &\, \cU_{\omega_1+\omega_2+\omega_5}&& H^\bullet(\cU_{\omega_1+\omega_2+\omega_5}(-1))=0.
  \end{align}
  Next, we compute the cohomology of $\cU_{\omega_1+\omega_2}
  \otimes  \cU_{\omega_1+\omega_2}^\vee$. We use
  the duality isomorphisms mentioned above and get: 
    \begin{align*} 
      \cU_{\omega_1+\omega_2} \otimes  \cU_{\omega_4+\omega_5}
      \simeq &\, \cO_X(3)  && \HH^\bullet(\cO_X) = \HH^0(\cO_X) \simeq \bC,\\
      \oplus &\, \cU_{\omega_2+2\omega_5}(1) && \HH^\bullet(\cU_{\omega_2+2\omega_5}(-2))=0,\\
      \oplus &\, \cU_{\omega_2+\omega_4}(1) &&
                                               \HH^\bullet(\cU_{\omega_2+\omega_4}(-2))
                                               =
                                               \HH^1(\cU_{\omega_2+\omega_4}(-2))
      \simeq V^{\omega_2},\\
      \oplus &\, \cU_{\omega_1+\omega_5}(2) ^{\oplus 2} && \HH^\bullet(\cU_{\omega_1+\omega_5}(-1))=0,\\
      \oplus &\, \cU_{\omega_1+\omega_2+\omega_4+\omega_5} && \HH^\bullet(\cU_{\omega_1+\omega_2+\omega_4+\omega_5}(-3))=0, \\
      \oplus &\, \cU_{2\omega_1+2\omega_5}(1) && \HH^\bullet(\cU_{2\omega_1+2\omega_5}(-2))=0,\\
      \oplus &\, \cU_{2\omega_1+\omega_4}(1) &&
                                                \HH^\bullet(\cU_{2\omega_1+\omega_4}(-2))=\HH^1(\cU_{2\omega_1+\omega_4}(-2))
                                                \simeq V^{2\omega_1}.
    \end{align*}

    Having computed this, we get that $\HH^i(Q^\vee \otimes
    \cU_{\omega_1+\omega_2}) = 0$ for all $i>0$ if and only if the
    boundary map induced by $\zeta$:
    \[
      V^{\omega_2}\oplus V^{2\omega_1} \simeq \HH^0(\cU_{\omega_1}^\vee \otimes \cU_{\omega_1+\omega_2}) \to 
      \HH^1(\cU_{\omega_1+\omega_2}^\vee \otimes \cU_{\omega_1+\omega_2}) 
      \simeq V^{\omega_2}\oplus V^{2\omega_1}
    \]
    is an isomorphism, and in this case
    $\HH^0(Q^\vee \otimes \cU_{\omega_1+\omega_2}) \simeq \bC$.
    In other words, 
    $Q$ is exceptional if and
    only if the following Yoneda map is an isomorphism:
    \begin{equation*}
      \Ext_X^1(\cU_{\omega_1+\omega_2},\cU_{\omega_1}) \otimes
      \Hom_X(\cU_{\omega_1}, \cU_{\omega_1+\omega_2}) \to  
      \Ext_X^1(\cU_{\omega_1+\omega_2}, \cU_{\omega_1+\omega_2}) .
    \end{equation*}
    In view of the isomorphisms above, this happens if and only if the cup-product maps below are isomorphisms:
    \begin{align*}
        H^1(\cU_{\omega_4-2\omega_6}) \otimes H^0(\cU_{\omega_2}) & \to H^1(\cU_{\omega_2+\omega_4-2\omega_6}) ,\\ 
    H^1(\cU_{\omega_4-2\omega_6}) \otimes H^0(\cU_{2\omega_1}) & \to H^1(\cU_{2\omega_1+\omega_4-2\omega_6}).
    \end{align*}

Fixing a Borel subgroup $B$ of $\Spin_{12}$, for a given $B$-dominant weight $\omega$, we consider the line bundle $\cL_\omega$ on the complete flag $W=\Spin_{12}/B$ 
and identify the bundle $\cU_\omega$ with the direct image of $\cL_\omega$ via the natural projection $W \to X$. Then our statement boils down to proving that the cup-product maps below are isomorphisms:
    \begin{align*}
        H^1(\cL_{\omega_4-2\omega_6}) \otimes H^0(\cL_{\omega_2}) & \to H^1(\cL_{\omega_4-2\omega_6} \otimes \cL_{\omega_2}), \\ 
    H^1(\cL_{\omega_4-2\omega_6}) \otimes H^0(\cL_{2\omega_1}) & \to H^1(\cL_{\omega_4-2\omega_6} \otimes \cL_{2\omega_1}).
    \end{align*}
However, this follows at once from the main theorem of \cite{dimitrov-roth}.

For what concern $\Ext_X(Q,Q(-1))$, it is sufficient to check that all irreducible bundles in $\cU_{\omega_1+\omega_2}^\vee \otimes \cU_{\omega_1}(-1)$, $\cU_{\omega_1+\omega_2}^\vee \otimes \cU_{\omega_1+\omega_2}(-1)$, $\cU_{\omega_1}^\vee \otimes \cU_{\omega_1}(-1)$, $\cU_{\omega_1}^\vee \otimes \cU_{\omega_1+\omega_2}(-1)$ have no non-vanishing cohomology (by BBW).
    \end{proof} 

\subsection{The exceptional Lefschetz collection}

Let us define the following collections of $\Spin_{12}$-homogeneous vector bundles
\begin{equation*}
  \begin{aligned}
    & \cA = \left( \cO_X, \cU_{\omega_1}, P, Q \right), \\
    & \cB = \left( \cO_X, \cU_{\omega_1}, P \right).
  \end{aligned}
\end{equation*}

\begin{lemma}
  The following is an exceptional collection in $\Db(X)$:
  \[
    \left(\cB, \cB(1), \dots, \cB(9)\right).
  \]
\end{lemma}

\begin{proof}
Recall that $\cO_X$ and $P$ are exceptional. We compute:
  \begin{equation} \label{15}
    \cU_{\omega_5} \otimes \cU_{\omega_1} \simeq  \cO_X(1) \oplus
                                                    \cU_{\omega_1+\omega_5}
  \end{equation}
  We write the largest intervals of integers where the twists of
  the bundles appearing in the right-hand-side have vanishing
  cohomology by BBW.
This gives:
  \begin{align*}
    \HH^\bullet(\cO_X(-t))& =0, && \mbox{for $t \in \{1,\ldots,9\}$},\\
    \HH^\bullet(\cU_{\omega_1+\omega_5}(-t))& =0, && \mbox{for $t
                                              \in \{1,\ldots,11\}$}, 
  \end{align*}
Then, using \eqref{basic-isos}, we get that $\cU_{\omega_1}$ is exceptional. We also get the required vanishing of twisted endomorphisms of $\cO_X$ and $\cU_{\omega_1}$. Also, we have the vanishing of $\Ext_X^\bullet(\cO_X(i),\cU_{\omega_1}(j))$ for $0 \le j < i \le 9$ and of  $\Ext_X^\bullet(\cU_{\omega_1}(i),\cO_X(j))$ for $0 \le j \le i \le 9$.

\medskip It remains to deal with $P$.
Looking at the extension defining $P$ and using BBW, we get $\Ext_X^\bullet(\cO_X(i),P(j))=0$ for $0 \le j < i \le 10$, so Serre duality ensures also $\Ext_X^\bullet(P(i),\cO_X(j))=0$ for $0 \le j \le i \le 9$. 

Next we show $\Ext_X^\bullet(\cU_{\omega_1}(i), P(j)) = 0$ for $0 \le j < i \le 9$ and note that the vanishing holds true even for $i=10$. We recall \eqref{basic-isos} and use $\cU_{\omega_5} \otimes \cU_{\omega_2} \simeq \cU_{\omega_1}(1) \oplus \cU_{\omega_2+\omega_5}$. Then, tensoring the sequence \eqref{defP} defining $P$ with $\cU_{\omega_5}(-1-t)$, for $1 \le t \le 9$, we get the desired vanishing by using:
\[
H^\bullet(\cU_{\omega_5}(-1-t))) = 
H^\bullet(\cU_{\omega_1}(-t))) = 
H^\bullet(\cU_{\omega_2+\omega_5}(-1-t))) = 0.
\]
Now Serre duality gives $\Ext_X^\bullet(P(i),\cU_{\omega_1}(j)) = 0$ for $0 \le j \le i \le 9$.

Finally we check $\Ext_X^\bullet(P(i), P(j)) = 0$ for $0 \le j < i \le 9$.
We compute:
\[
\cU_{\omega_2} \otimes \cU_{\omega_4} \simeq \cO_X(2) \oplus \cU_{\omega_1+\omega_5}(1) \oplus \cU_{\omega_2+\omega_4}.
\]

Tensoring the sequence \eqref{defP} defining $P$ with its dual and using \eqref{basic-isos}, we deduce the desired vanishing results from the following ones, which in turn are given by BBW for $1 \le t \le 9$:
\[
H^\bullet(\cO_X(-t)) = H^\bullet(\cU_{\omega_2}(-t)) = H^\bullet(\cU_{\omega_4}(-2-t)) = H^\bullet(\cU_{\omega_2+\omega_4}(-t)) = H^\bullet(\cU_{\omega_1+\omega_5}(-1-t)) = 0.
\]
\end{proof}

\begin{lemma} \label{it's exceptional}
  The following is an exceptional collection in $\Db(X)$:
  \begin{equation}\label{eq:collection-d6-p6}
    \left(\cA, \cA(1), \cB(2) \dots, \cB(9)\right).
  \end{equation}
\end{lemma}

\begin{proof}
By the previous lemma and thanks to Serre duality, we will be done once we prove $\Ext_X^\bullet(Q,Q(-1))=0$ (which we did in Lemma \ref{lemma:canonical-extensionQ}) and:
  \[
  \Ext_X^\bullet(\cO_X,Q(-t))=\Ext_X^\bullet(\cU_{\omega_1},Q(-t))=
  \Ext_X^\bullet(P,Q(-t))=0,
  \]
  for $1 \le t \le 10$.
Looking at the sequence \eqref{defQ}  defining $Q$, we see that BBW directly implies $H^\bullet(Q(-t))=0$ for $1 \le t \le 10$.
As for $\Ext_X^\bullet(\cU_{\omega_1},Q(-t))=0$, note that the case $t=10$ is \eqref{Q vs U*} by Serre duality. On the other hand, for $1 \le t \le 9$, this follows from \eqref{basic-isos} and from the vanishing 
\[
H^\bullet(\cU_{\omega_5} \otimes \cU_{\omega_1}(-1-t))=0, \quad 
H^\bullet(\cU_{\omega_5} \otimes \cU_{\omega_1+\omega_2}(-1-t))=0,
\]
for $1 \le t \le 9$, which in turn is a consequence of \eqref{125 I}, \eqref{125 II}, \eqref{125 III} and \eqref{15}. 

Finally, let us show that $\Ext_X^\bullet(P,Q(-t))=0$. For $t \ne 10$, this follows from BBW, \eqref{basic-isos} and from the isomorphisms:
\begin{align*}
\cU_{\omega_4} \otimes \cU_{\omega_1} & \simeq \cU_{\omega_5}(1) \oplus \cU_{\omega_1+\omega_4}, \\
\cU_{\omega_4} \otimes \cU_{\omega_1+\omega_2} & \simeq \cU_{\omega_1}(2) \oplus \cU_{\omega_2+\omega_5}(1) \oplus
\cU_{\omega_1+\omega_2+\omega_4} \oplus \cU_{2\omega_1+\omega_5}.
\end{align*}
For $t=10$, the statement is equivalent to $\Ext_X^\bullet(Q,P)=0$. To check this last vanishing, using the isomorphisms of the previous display, we are reduced to show $\Ext_X^\bullet(Q,\cU_{\omega_2})=0$ and in turn to show that cupping with $\zeta \in \Ext_X^1(\cU_{\omega_1+\omega_2},\cU_{\omega_1}) \simeq H^1(\cU_{\omega_4}(-2)) \simeq \bC$ induces an isomorphism:
\[
V^{\omega_1} \simeq H^0(\cU_{\omega_1}) \simeq \Hom_X(\cU_{\omega_1},\cU_{\omega_2}) \to \Ext_X^1(\cU_{\omega_1+\omega_2},\cU_{\omega_2}) \simeq H^1(\cU_{\omega_1+\omega_4}(-2)) \simeq V^{\omega_1}
\]

As in the proof of Lemma \ref{lemma:canonical-extensionQ}, we move to the complete flag $W$, so that showing this non-degeneracy amounts to checking that the cup-product map below is an isomorphism:
\[
H^1(\cL_{\omega_4-2\omega_6}) \otimes H^0(\cL_{\omega_1}) \to H^1(\cL_{\omega_4-2\omega_6} \otimes \cL_{\omega_1}),
\]
and this follows from \cite{dimitrov-roth}.
\end{proof}

Define the full triangulated subcategory
\begin{equation*} 
  \cD = \langle \cA, \cA(1), \cB(2), \dots, \cB(9) \rangle.
\end{equation*}
We are going to show that \eqref{eq:collection-d6-p6} is full, i.e. we have
\begin{equation*}
  \cD^\perp = 0.
\end{equation*}

\section{Warming up for fullness}
\label{warming up}

Let $\cD$ be the full triangulated subcategory of $\Db(X)$ generated by our exceptional
collection, i.e. we define
\begin{equation*}
  \cD = \big \langle \cA, \cA(1), \cB(2), \cB(3), \cB(4) \big \rangle \subset \Db(X).
\end{equation*}
Thus, we have a semiorthogonal decomposition
\begin{equation*}
  \Db(X) = \langle \cD^\perp, \cD \rangle.
\end{equation*}
Our aim is to prove $\cD^\perp = 0$. To achieve this, we will restrict to a covering family of smaller spinor varieties whose derived category is well-known and prove that any object orthogonal to $\cD$ restricts to zero over such varieties by showing that the structure sheaf of these subvarieties is resolved by objects in $\cD$.
After describing such covering family, we will show fullness in the easier and well known case of spinor 10-folds as a warming up. In doing so, we will make use of an exact complex appearing in Section \ref{section:dissecting}, pointing out that the existence of such complex is fundamental for our proof of fullness.

\subsection{A covering family of spinor varieties}

Let us come back to the general case of a vector space $V$ of dimension $2n$. Let
\begin{equation*}
  q \colon V \times V \to \bC
\end{equation*}
be the symmetric bilinear form defining $X:=\Spin_{2n}/\rP_n$, i.e. we have
\begin{equation*}
  X = \OG_+(n,V).
\end{equation*}
Recall that we have
\begin{equation*}
  H^0(X, \cU^\vee) = V^\vee \overset{\simeq}{\to} V.
\end{equation*}
Since $q$ is non-degenerate, there is a bijection between elements $w \in V$ and
sections $s_w \in H^0(X, \cU^\vee) = V^\vee$ that sends $w$ to $s_w = q(w, \quad )$.
It is easy to see that we have
\begin{equation}\label{eq:non-degenerate-not-on-the-quadric}
  q_{\vert W} \, \text{  is non-degenerate} \, \iff \, q(w,w) \neq 0,
\end{equation}
where $W = \ker s_w$. If $s_w$ satisfies \eqref{eq:non-degenerate-not-on-the-quadric},
then we can define two things:
\begin{enumerate}
  \item A morphism of algebraic varieties
  \begin{equation}\label{eq:identification-og-5-11}
    \begin{aligned}
      \varphi_w \colon \OG_+(n,V) & \to \OG(n-1,W) \\
      U & \mapsto U \cap W,
    \end{aligned}
  \end{equation}
  which is an isomorphism.

  \smallskip

  \item The natural morphism $s_w \colon \cO_X \to \cU^\vee$ does not vanish anywhere (since there are no $n$-dimensional isotropic subspaces in $W$) and, therefore,
  defines a short exact sequence of vector bundles
  \begin{equation}\label{eq:non-degenerate-subspace-exact-sequence}
    0 \to \cO_X \overset{s_w}{\to} \cU^\vee \to \cE^\vee \to 0,
  \end{equation}
  where $\cE^\vee$ is a vector bundle of rank $n-1$ with $
    H^0(X, \cE^\vee) = W^\vee$.
\end{enumerate}

Let us fix a section $s_w$ satisfying \eqref{eq:non-degenerate-not-on-the-quadric}.

\begin{lemma}\label{lemma:auxiliary-1}
  Let $s \in H^0(X, \cE^\vee) = W^\vee$ and consider the zero-locus $Y_s$ of $s$.
  \begin{enumerate}[label=\roman*)]
    \item \label{cond-i} If $s$ is general enough, then
    $Y_s \simeq \OG(n-1,2n-2).$
    Let us denote the inclusion by
    $$i_s \colon Y_s \to X.$$
    \item For any section $s$ as in \ref{cond-i}, we have
    \begin{equation*}
      i^* \cU^\vee = \cU_{n-1}^\vee \oplus \cO_{Y_s}
    \end{equation*}

    \item Varying $s \in H^0(X, \cE^\vee)$ as above we can cover $X$ by copies of $\OG(n-1,2n-2)$.

    \item If for any object $F \in \Db(X)$ the restrictions $i_s^* F$ vanish for all
    $s \in H^0(X, \cE^\vee)$ as above, then $F = 0$.
  \end{enumerate}
\end{lemma}

\begin{proof}
  For $s$ to be \textit{general enough}, it is enough to
    satisfy the analogue of \eqref{eq:non-degenerate-not-on-the-quadric}. i.e.
    the restriction of $q$ to $L:= \ker(s) \subset W$ should be non-degenerate.
      
  \begin{enumerate}[label=\roman*)]
\item Under \eqref{eq:identification-og-5-11} the vector bundle $\cE^\vee$ corresponds
    to the dual of the tautological subbundle on $\OG(n-1,W)$. Hence, we get the claim.

    \item Under \eqref{eq:identification-og-5-11} the sequence
    \eqref{eq:non-degenerate-subspace-exact-sequence} shows that there is a non-trivial
    extension between the dual of the tautological subbundle and the structure sheaf
    on $\OG(n-1,W)$. However, by BBW on $\OG(5,L) = \OG(n-1,2n-2)$ such extensions vanish and
    the sequence splits.

    \item Indeed, for any $n-1$-dimensional isotropic subspace $U_{n-1} \subset W$
    we can consider $U_{n-1}^\perp$, take any element $u \in  U_{n-1}^\perp \setminus U_{n-1}$
    and take $L = u^\perp$. Clearly we have $U_{n-1} \subset L$ and $q_{\vert L}$
    is non-degenerate by \eqref{eq:non-degenerate-not-on-the-quadric}.

    \item This is \cite[Lemma 4.5]{Ku08a}.
  \end{enumerate}
\end{proof}

\subsection{Full exceptional collection on the spinor 10-fold}

The orthogonal Grassmannian $Y = \OG(5,10)$ has two connected components that we
denote by
\begin{equation*}
  Y_- = \OG_-(5,10) \quad \text{and} \quad Y_+ = \OG_+(5,10).
\end{equation*}
These components are isomorphic to each other, we call them spinor 10-folds.

As usual, on $Y = \OG(5,10)$ we can consider the tautological subbundle $\cU_5$ of rank $5$.
We denote as
$
  \cU_{5,\pm} \coloneqq {\cU_5}\vert_{ Y_{\pm}}
$
its restrictions to $Y_{\pm}$ 

We prove the following result as a useful warm-up to the case of $X=\Spin_{12}/\rP_6=\OG_+(6,12)$.
\begin{theorem}
  \label{theorem:collection-og-5-10}
  We have
  \begin{equation*}
    \Db(Y_{\pm}) = \Big\langle \cO, \cU^\vee_{5,\pm}, \cO(1), \cU^\vee_{5,\pm}(1), \dots ,
    \cO(7), \cU^\vee_{5,\pm}(7) \Big\rangle
  \end{equation*}
\end{theorem}

\begin{proof}
Let us fix the $+$ sign and let us define $\cD_5:=\langle \cO, \cU^\vee_{5}, \cO(1), \cU^\vee_{5}(1), \dots ,
    \cO(7), \cU^\vee_{5}(7) \rangle$ on $Y:=Y_+$. Let us take an object $F \in \cD_5^\perp$, i.e. we have
  \begin{equation*}
    \Ext_X^\bullet(A,F) = 0 \quad \text{for any} \quad A \in \cD_5.
  \end{equation*}
  Let $s \in H^0(X,\cE^\vee)$ be a general section and $\is \colon Z_s \to Y$ the
  embedding of its zero locus, as in Lemma \ref{lemma:auxiliary-1}(1).

  Let us consider the set of vector bundles on Y defined by
  \begin{equation*}
    \Sigma_5 \coloneqq \{ \cO_Y(t) \, \vert \, t \in [2,7] \}\cup \{\cU^\vee_5(2)\}.
  \end{equation*}
  Let us denote by $\cE_5^\vee$ the vector bundle defined in \eqref{eq:non-degenerate-subspace-exact-sequence}. We claim that for any $E \in \Sigma_5$ and any $j$ the bundle
  $E \otimes \wedge^j \cE_5^\vee$ lies in $\cD_5$. Let us for the moment assume that the claim is true. Then we have
  \begin{equation*}
    \Ext_X^\bullet(E \otimes \wedge^j \cE_5^\vee, F) =
    H^\bullet(Y, \wedge^j \cE_5 \otimes E^\vee \otimes F) = 0 \quad \text{for all } j,
  \end{equation*}
  and making use of the Koszul complex
  \begin{equation*}
    0 \to \wedge^4 \cE_5 \to \dots \to \cE_5 \to \cO_Y \to \is_* \cO_{Z_s} \to 0,
  \end{equation*}
  we obtain
  \begin{equation*}
    H^\bullet(Y, \left( E^\vee \otimes F \right) \otimes \is_* \cO_Z) = 0.
  \end{equation*}
  Now, by projection formula we rewrite
  \begin{equation*}
    H^\bullet(Y, \left( E^\vee \otimes F \right) \otimes \is_* \cO_Z) = H^\bullet(Z_s, \is^* \left( E^\vee \otimes F \right))
    = \Ext^\bullet_Z(\is^* E , \is^* F) = 0.
  \end{equation*}

  Recall that $Z_s \simeq \OG(4,8)$ has two connected components $Z_{s+}$ and $Z_{s-}$ which are two six dimensional quadrics.
  We denote the compositions $Z_{s\pm} \subset Z_s \overset{\is}{\to} Y$ by $\is_{\pm}$.
  Using this notation we have
  \begin{equation*}
    \Ext^\bullet_{Z_s}(\is^* E , \is^* F) =
    \Ext^\bullet_{Z_{s+}}(\is_+^* E , \is_+^* F) \oplus \Ext^\bullet_{Z_{s-}}(\is_-^* E , \is_-^* F).
  \end{equation*}
  Hence, we have
  \begin{equation*}
    \Ext^\bullet_{Z_{s+}}(\is_+^* E , \is_+^* F) = 0
    \quad \text{and} \quad
    \Ext^\bullet_{Z_{s-}}(\is_-^* E , \is_-^* F) = 0.
  \end{equation*}
  Applying Lemma \ref{lemma:auxiliary-1}(2) and the fact that the six dimensional quadrics $Z_{s\pm}$ admit the following full exceptional collection (see \cite{Ka}) 
  $$\Db(Z_{s\pm})= \Big\langle \cO(2), \cU^\vee_{4,\pm}(2), \cO(3), \dots ,
    \cO(7) \Big\rangle,$$
  we obtain
  \begin{equation*}
    \is_+^* F = 0 \quad \text{and} \quad \is_-^* F = 0.
  \end{equation*}
  Hence, we conclude $\is^* F = 0$. Finally, since the above argument works for any general $s \in H^0(Y,\cE_5^\vee)$,
  by Lemma \ref{lemma:auxiliary-1}(3,4) we obtain $F = 0$. 
  
  Now, let us prove the claim. We need to prove that $\wedge^j \cE_5^\vee(t)\in \cD_5$ for $t\in[2,7]$ and $\cU_5^\vee\otimes \wedge^j \cE_5^\vee (2)\in \cD_5$ for all possible $j$'s. From the exact sequence
  $$ 0 \to \cO_Y {\to} \cU_5^\vee \to \cE_5^\vee \to 0$$
  we deduce that our claim is implied by the fact that $\wedge^j \cU_5^\vee(t)\in \cD_5$ for $t\in[2,7]$ and $\cU_5^\vee\otimes \wedge^j \cU_5^\vee (2)\in \cD_5$ for all possible $j$'s. 
  
  The bundles $\cO(t)$ and $\cU^\vee_5(t)$ for $t\in [0,7]$ generate $\cD_5$. From the exact sequence $0\to \cU_5 \to V_{10}\otimes\cO_Y \to \cU^\vee_5 \to 0$, where $V_{10}$ is a ten dimensional vector space, we deduce that $\cU_5(t)\in \cD_5$ for $t\in [0,7]$. Thus $\cO_Y(t),\cU_5(t),\wedge^4 \cU_5(t)=\cU^\vee(t-2),\wedge^5\cU_5(t)=\cO_Y(t-2)$ all belong to $\cD_5$ for $t\in[2,7]$. 
  
  Recall that $\ss(S^-\otimes \cO_Y)=\cU_5(-1)\oplus  \wedge^3 \cU_5^\vee(-1)\oplus \cO_Y(1)$, so we deduce that $\wedge^3 \cU_5^\vee(t-2)=\wedge^2 \cU(t)\in \cD_5$ for $t\in [2,7]$. Similarly the fact that $\ss(S^+\otimes \cO_Y)=\cO_Y(-1)\oplus  \wedge^2 \cU_5^\vee(-1)\oplus \cU_5(1)$ implies that $\wedge^2 \cU_5^\vee(t-2)=\wedge^3 \cU(t)\in \cD_5$ for $t\in [2,7]$.
  
  Now we need to deal with $\cU_5^\vee \otimes \wedge^j \cU_5(2)$. When $j=0$ and $j=5$, $\cU_5^\vee \otimes \wedge^j \cU_5(2)\in \cD_5$. For $j=1$ use the decomposition $\ss(\wedge^2 V_{10}\otimes \cO_Y)=\wedge^2 \cU_5 \oplus \cU_5\otimes \cU_5^\vee \oplus \wedge^2 \cU_5^\vee$ to deduce that $\cU_5\otimes \cU_5^\vee(t)\in \cD_5$ for $t\in[2,5]$. For $j=3$ use the decomposition $\ss(S^+\otimes \cU^\vee_5)=\cU_5^\vee(-1)\oplus \wedge^3 \cU_5 \otimes \cU_5^\vee(1)\oplus \cU_5\otimes \cU_5^\vee(1)$ to deduce that $\wedge^3\cU_5\otimes \cU_5^\vee(t)\in \cD_5$ for $t\in[2,5]$. The cases $j=2$ and $j=4$ can be dealt with in parallel. Indeed one can use the decomposition $\ss(S^-\otimes \cU^\vee_5)=\cU_5^\vee(1)\oplus \wedge^2 \cU_5 \otimes \cU_5^\vee(1)\oplus \wedge^4 \cU_5\otimes \cU_5^\vee(1)$ to deduce that, if $t\in [1,6]$ then: $\wedge^2\cU_5\otimes \cU_5^\vee(t)\in \cD_5$ if and only if $\wedge^4\cU_5\otimes \cU_5^\vee(t)\in \cD_5$. 
  
  Finally, we want to prove for instance that $\wedge^2 \cU_5 \otimes \cU_5^\vee(2)\in \cD_5$. For this the exact complex appearing in Proposition \ref{prop_complex_5} is crucial (Section \ref{sec_spinor_bundles} is independent of this proof, so we can use the results therein). Indeed from that complex one deduces that $R_5(t)\in \cD_5$ for $t\in[0,5]$, which in turn implies that $\cU_5^{\omega_1+\omega_2}(t)\in \cD_5$ for $t\in[0,3]$. Then, using the decomposition $\ss(\wedge^2 \cU^\vee_5 \otimes V_{10})=\cU_5^{\omega_1+\omega_2} \oplus \wedge^3 \cU_5^\vee \oplus \cU_5^\vee \oplus \cU_5^{\omega_2+\omega_4}(-2)$, we obtain that $\cU_5^{\omega_2+\omega_4}(t)\in \cD_5$ for $t\in[-2,1]$. Notice also that previously we showed that $\cU_5\otimes \cU_5^\vee(t)=\cO_Y(t)\oplus \cU_5^{\omega_1+\omega_4}(t-2) \in \cD_5$ for $t\in [2,5]$, so $\cU_5^{\omega_1+\omega_4}(t) \in \cD_5$ for $t\in [0,3]$. These two facts imply that $\cU_5^{\omega_1+\omega_3}(t)\in \cD_5$ for $t=0,1$ because of the decomposition $\ss(\cU_5^{\omega_1+\omega_4}(-2)\otimes V_{10})=\cU_5^{\omega_1+\omega_3}(-2)\oplus \cU_5^{\omega_2+\omega_4}(-2)\oplus \cU_5\oplus \cU_5^\vee$. Then from the decomposition $\ss(\wedge^2 \cU_5 \otimes \cU_5^\vee)=\cU_5^{\omega_1+\omega_3}(-2)\oplus \cU_5$ we deduce that $\wedge^2 \cU_5 \otimes \cU_5^\vee(t)\in \cD_5$ for $t=2,3$.
\end{proof}

\begin{remark}
Theorem \ref{theorem:collection-og-5-10} was already known from \cite[Section 6.2]{K06}. The proof given here is more direct and corresponds better to our approach. 
\end{remark}

\section{Dissecting Spin bundles}
\label{section:dissecting}

In this section, we look more closely to the vector bundles on $X=\OG(5,12)_+$ induced by the spinor representations. The main goal is to prove Proposition \ref{prop_complex_5} and 
\ref{prop_complex_6}, which in turn will be used in Section \ref{more} in view of showing fullness of our collection.
We ofter abbreviate $\cO_X$ to $\cO$.

\subsection{The Spin representations}

In the following we will recall a selection of generalities about the Clifford algebra and Spin representations that can be found, for instance, in \cite{EMClif}. Let us begin with an even dimensional vector space $V$ endowed with a non-degenerate symmetric form $q$. The Clifford algebra is defined as the quotient of the tensor algebra $V^\otimes$ by all relations of the form $v\otimes v - q(v,v)$ for $v\in V$. Notice that both $V$ and $\wedge^2 V \simeq \mathfrak{so}_V$ embed inside the Clifford algebra. 

\subsubsection{The Spin representation and exterior powers}

Let us fix a maximal isotropic subspace $U$ of $V$. Any other maximal isotropic subspace intersecting $U$ transversally can be identified through $q$ with $U^\vee$; we thus get a decomposition of $V=U\oplus U^\vee$. The Spin representations can be identified, as vector spaces, as follows:
\begin{align*}
    S^+&:=\wedge^+ U^\vee=\bigoplus_i \wedge^{2i} U^\vee, \\
    S^-&:=\wedge^- U^\vee=\bigoplus_i \wedge^{2i+1} U^\vee. 
\end{align*}
There is a natural action $$\eta_\pm:=V\otimes S^{\pm} \to S^{\mp}$$ defined as follows: $\eta_{\pm}(v\otimes \omega)=v\wedge \omega$ if $v\in U^\vee$ and $\eta_{\pm}(v\otimes \omega)=v\lrcorner \omega$ if $v\in U$, where $\lrcorner$ is the contraction. This induces an action of the Clifford algebra, and hence of $\mathfrak{so}_V$, on $S^{\pm}$, which endows this vector space with a structure of $\Spin$-representation; $S^\pm$ are the so-called Spin representations. It turns out that, if the dimension of $V$ is $2n$, then $(S^\pm)^\vee=S^{(-1)^n \pm}$ as representations. Moreover if $n$ is odd then $S^+=V^{\omega_5}$ and $S^-=V^{\omega_6}$ while if $n$ is even then $S^+=V^{\omega_6}$ and $S^-=V^{\omega_5}$. Notice moreover that the action $\eta_\pm$ naturally induces a $\Spin$-equivariant morphism $\eta_\pm^{\otimes i}:V^{\otimes i} \otimes S^\pm \to S^{(-1)^{i}\pm}$, and thus also a $\Spin$-equivariant morphism 
$$\wedge^i \eta_\pm:\wedge^i V \otimes S^\pm \to S^{(-1)^{i}\pm}.$$
In the following we want to use the morphism $\wedge^i \eta_\pm$ to construct some exact complexes on $\Spin_{2n}/\rP_n$ for $n=5,6$. Before doing so, we will recall basic linear algebra facts in order to explain how to rewrite $\wedge^i \eta_\pm$ as a morphism $\xi:S^\pm \otimes (S^{(-1)^i\pm})^\vee \to \wedge^i V^\vee \simeq \wedge^i V$.

\subsubsection{Linear algebra digression}

Let us begin with a linear morphism $u:A\otimes B \to C$ for three vector spaces $A,B,C$. This means that $u\in A\otimes B \otimes C^\vee=\Hom(A\otimes C^\vee,B^\vee)$, so it defines another morphism $t:A\otimes C^\vee \to B^\vee$. Clearly one can recover $u$ from $t$ as well. 

\begin{lemma}
\label{lem_lin_algebra}
$\Imm(u)^\perp$ is identified with the subspace $\{x\in C^\vee \mid t(a,x)=0 \mbox{ }\forall a\in A \}\subset C^\vee$.
\end{lemma}

\begin{proof}
Let us denote by $D$ the above subspace. By definition of $t$, for any $x\in C^\vee$, $a\in A$ and $b\in B$, $x(u(a\otimes b))=t(a\otimes x)(b)$. It is straightforward to deduce that $x\in \Imm(u)^\perp$ if and only if $x\in D$.
\end{proof}

\subsection{Spinor bundles}
\label{sec_spinor_bundles}

Let us consider the variety $\Spin_{2n}/\rP_n=\OG_+(n,2n)$ which is one of the two isomorphic connected components of the variety parametrizing maximal isotropic subspaces of $V$. Let us denote by $\epsilon:= (n \mod 2)$. The line bundle $\cO(1)=\cU_{\omega_n}$ embeds $\Spin_{2n}/\rP_n$ inside $\bP(V^{\omega_n})=\bP(S^{(-1)^\epsilon})$. Thus $\cO(1)$ is a $G$-equivariant quotient of $S^{(-1)^\epsilon}\otimes \cO$. In fact, one can construct a filtration of $G$-equivariant vector bundles $$0=:\cF_{0} \subset \cF_{1} \subset \cdots \subset \cF_{\lfloor n \rfloor +\epsilon}:=S^{(-1)^\epsilon}$$ such that $$\cF_{i+1}/\cF_{i}=(\wedge^{2i+\epsilon} \cU^\vee) (-1).$$ This is the relative version of the filtration of $S^{(-1)^\epsilon}=\wedge^+ U^\vee$ given by the subspaces $F_{i+1}:=\sum_{j\leq i}\wedge^{2j+\epsilon}U^\vee$. A similar filtration exists for $S^{-(-1)^\epsilon}\otimes \cO$. For instance, we get that $\cF_1=\cU^\vee(-1)$ is a subbundle of $S^-\otimes \cO$. This filtration was described in \cite[Proposition 6.3]{Ku08a}.

\subsubsection{An exact complex in low dimension} 

We will now construct an exact complex of vector spaces using the morphisms $\wedge^i \eta_\pm$ when $n=5$ and $n=6$. We believe that this type of complexes can be generalized for higher $n$ and will be crucial in proving fullness of exceptional collections on $\Spin_{2n}/\rP_n$ for higher $n$. From now on we fix $\eta:=\eta_+$

\subsubsection{The case $n=5$}

In this case we have the following decomposition of representations: $S^+\otimes S^-=\bC \oplus \wedge^2 V \oplus V^{\omega_4+\omega_5}$. This implies that there exists a unique $G$-equivariant morphism $S^+\otimes S^- \to \wedge^2 V$, which must then be equal to $\wedge^2 \eta$ (notice that $(S^-)^\vee=S^+$ since $n$ is odd). As a consequence of BBW $H^0(\cU(1) )=S^-$ and thus there exists a unique $G$-equivariant morphism $\cU^\vee(-1)\otimes S^+ \to \wedge^2 V$. This morphism must then be the composition $\wedge^2 \eta \circ (i\otimes \id)$ where $i$ is the inclusion $i:\cU^\vee(-1)\to S^-\otimes \cO$. All in all we get a $G$-equivariant morphism $\wedge^2 \eta \circ (i\otimes \id):\cU^\vee(-1)\otimes S^+ \to \wedge^2 V$. The aim of this section is to prove the following:

\begin{proposition}
\label{prop_complex_5}
There exists a $G$-equivariant extension $$0\to \cU^\vee(-2)\to R_5\to \cU_{\omega_1+\omega_2}(-2) \to 0$$ and a $G$-equivariant exact complex of vector bundles 
$$ 0\to R_5 \to \cU^\vee(-1)\otimes S^+ \to \wedge^2 V\otimes \cO \to \wedge^2 \cU^\vee \to 0 ,$$
where the central map is $\wedge^2 \eta \circ(i\otimes \id)$.
\end{proposition}

\begin{proof}
The morphism $\wedge^2 V\otimes \cO \to \wedge^2 \cU^\vee $ above is the natural projection induced by the exact sequence 
$$0\to \cU \to V\otimes \cO \to \cU^\vee \to 0.$$ 
Since this map, as well as $\wedge^2 \eta \circ (i\otimes \id)$, is a $G$-equivariant morphism of $G$-homogeneous vector bundles, it is sufficient to restrict to any fiber of $\Spin_{10}/\rP_5$ to prove exacteness. More precisely we will show that, if $[U]\in \Spin_{10}/\rP_5$, the induced complex of vector spaces
\begin{equation} 
\label{eq_complex_5_proof}
(\cU^\vee(-1)\otimes S^+)|_{[U]} \to (\wedge^2 V\otimes \cO)|_{[U]} \to (\wedge^2 \cU^\vee)|_{[U]} \to 0
\end{equation}
is exact. From this it will follow that the complex $$ \cU^\vee(-1)\otimes S^+ \to \wedge^2 V\otimes \cO \to \wedge^2 \cU^\vee \to 0 $$ is exact. The result will then follow by noticing that, since $\ss(\cU^\vee(-1)\otimes S^+)=\cU^\vee(-2)\oplus \cU_{\omega_1+\omega_2}(-2) \oplus \wedge^3 \cU^\vee(-2) \oplus \mathfrak{sl}(\cU)\oplus \cO$ and $\ss(\wedge^2 V)=\wedge^2 \cU \oplus \cO \oplus \mathfrak{sl}(\cU)\oplus \wedge^2 \cU^\vee$, the semisimple reduction of the kernel of $\wedge^2 \eta \circ (i\otimes \id)$ is necessarily equal to $\cU^\vee(-2) \oplus \cU_{\omega_1+\omega_2}(-2)$.
\medskip

Let $[U]\in \Spin_{10}/\rP_5$ be any point. Then $(\cU^\vee(-1))|_{[U]}\simeq U^\vee$ is a subspace of $(S^-\otimes \cO)|_{[U]}=S^-=\bigoplus_i \wedge^{2i+1}U^\vee$ - here the last equality only holds as an equality of $L(\rP_5)$-representations. 
Following the linear algebra digression, the morphism $\xi \circ (i\otimes \id)$ corresponds to the morphism $\wedge^i \eta \circ (i\otimes \id)$. Moreover, letting $$t:=(\wedge^2 \eta \circ (i\otimes \id))_{[U]}: U^\vee \otimes \wedge^2 V^\vee \to S^-=\bigoplus_i \wedge^{2i+1}U^\vee $$ and $$u:=(\xi\circ (i\otimes \id))_{[U]}:U^\vee \otimes S^+ \to \wedge^2 V$$ and applying Lemma \ref{lem_lin_algebra}, we deduce the following:
$$\Imm(u)^\perp=\{ v\in \wedge^2 V^\vee \mid \forall f\in U^\vee \subset S^-=\bigoplus_i \wedge^{2i+1}U^\vee,\mbox{ } t(f\otimes v)=0 \}\subset \wedge^2 V^\vee.$$
We bothered going through all of this because we have a very explicit description of the map $t$, which is the one induced by $\eta$; let us see how to use it. First notice that $t$ is a $\rP_5$-equivariant morphism, so in particular let us treat it as a $L(\rP_5)$-equivariant morphism. Thus we can decompose $\wedge^2 V=\wedge^2 U \oplus (U\otimes U^\vee) \oplus \wedge^2 U^\vee=\wedge^2 U \oplus \bC \oplus \mathfrak{sl}(U) \oplus \wedge^2 U^\vee$. By $L(\rP_5)$-equivariance, each of these factors is either completely contained in $\Imm(u)^\perp$ or does not intersect non-trivially $\Imm(u)^\perp$. In order to distinguish the two cases it is thus sufficient to decide whether a non-zero vector in a given factor belongs to $\Imm(u)^\perp$ or not. We thus have four cases to deal with. We will denote by $u_1,\dots,u_5$ a basis of $U$ and by $w_1,\dots,w_5$ the dual basis. We will denote by $u_{ij}=u_i\wedge u_j$ and by $w_{ij}=w_i\wedge w_j$; $\delta_{i,j}$ will denote Kronecker's delta.
\begin{itemize}
    \item[$\wedge^2 U$:] Let $0\neq u_{ij}\in \wedge^2 U$ and $w_k\in U^\vee$. Then $t(w_k\otimes u_{ij})=u_i\lrcorner(u_j\lrcorner w_k)-u_j\lrcorner(u_i\lrcorner w_k)=0,$ for any $k=1,\dots,5$, so $\wedge^2 U\subset \Imm(u)^\perp$.
    \item[$\bC$:] Let $0\neq \sum_i u_{i}\wedge w_i\in \bC\subset \wedge^2 V$ and $w_k\in U^\vee$. Then $t(w_k\otimes (\sum_iu_i\wedge w_i))=\sum_i(u_i\lrcorner (w_{ik}) - \delta_{i,k}w_k)=\sum_i((1-\delta_{i,k})w_k-\delta_{i,k}w_k)=\sum_i(1-2\delta_{i,k})w_k=2w_k\neq 0$, so $\bC\cap \Imm(u)^\perp=0$.
    \item[$\mathfrak{sl}(U)$:] Let $0\neq u_{i}\wedge w_j\in \mathfrak{sl}(U)$ for $i\neq j$, and $w_k\in U^\vee$. Then $t(w_k\otimes u_i\wedge w_j)=u_i\lrcorner (w_{jk})- (u_i\lrcorner w_k)w_j=-2\delta_{i,k}w_k\neq 0$, so $\mathfrak{sl}(U) \cap \Imm(u)^\perp=0$.
    \item[$\wedge^2 U^\vee$:] Let $0\neq w_{ij}\in \wedge^2 U^\vee$ and $w_k\in U^\vee$. Then $t(w_k\otimes w_{ij})=w_{ijk}\neq 0$, so $\wedge^2 U^\vee \cap \Imm(u)^\perp=0$.
\end{itemize}
The previous computations imply that $\Imm(u)^\perp=\wedge^2 U\subset \wedge^2 V \simeq \wedge^2 V^\vee$. This is equivalent to the fact that $\Imm(u)$ is the kernel of $\wedge^2 V \to \wedge^2 U^\vee$. Moreover the latter morphism is clearly surjective, so we deduce that the complex in \eqref{eq_complex_5_proof} is exact. The statement of the proposition follows.
\end{proof}

\subsubsection{The case $n=6$}

In this case we have the following decomposition of representations: $S^+\otimes S^-=V \oplus \wedge^3 V \oplus V^{\omega_5+\omega_6}$. This implies that there exists a unique $\rG$-equivariant morphism $S^+\otimes S^- \to \wedge^3 V$, which must then be equal to $\wedge^3 \eta$ (notice that $(S^+)^\vee=S^+$ since $n$ is even). As a consequence of the BBW Theorem $H^0(\cU(1) )=S^-$ and thus there exists a unique $G$-equivariant morphism $\cU^\vee(-1)\otimes S^+ \to \wedge^3 V$. This morphism must then be the composition $\wedge^3 \eta \circ (i\otimes \id)$ where $i$ is the inclusion $i:\cU^\vee(-1)\to S^-\otimes \cO$. All in all we get a $\rG$-equivariant morphism $\wedge^3 \eta \circ (i\otimes \id):\cU^\vee(-1)\otimes S^+ \to \wedge^3 V$. The aim of this section is to prove an analogue of Proposition \ref{prop_complex_5}. In order to do so, let us begin by defining the vector bundle $T$ as the cokernel of the unique $\rG$-equivariant inclusion $\cU^\vee \to V\otimes \wedge^2 \cU^\vee$; we thus have an exact sequence
$$ 0 \to \cU^\vee \to V\otimes \wedge^2 \cU^\vee \to T \to 0.$$
\begin{proposition}
\label{prop_complex_6}
There exists a $\rG$-equivariant extension $R_6$ whose semisimple reduction is $$\ss(R_6)=\cU^\vee(-2)\oplus \cU_{\omega_1+\omega_2}(-2)$$ and a $\rG$-equivariant exact complex of vector bundles 
$$ 0\to R_6 \to \cU^\vee(-1)\otimes S^+ \to \wedge^3 V\otimes \cO \to T \to \cU_{\omega_1+\omega_2} \to 0 ,$$
where the second map is $\wedge^3 \eta \circ(i\otimes \id)$.
\end{proposition}


\begin{proof}
The morphism $\wedge^3 V\otimes \cO \to T$ is the unique $\rG$-equivariant morphism and it is the one induced on the quotient from the natural one $\wedge^3 V\otimes \cO \to V\otimes \wedge^2 \cU^\vee $. Since the cokernel of the latter is $\cU_{\omega_1+\omega_2}$, this is also the cokernel of the former. Since the morphism $\wedge^2 \eta \circ (i\otimes \id)$ is a $\rG$-equivariant morphism of $\rG$-homogeneous vector bundles, it is sufficient to restrict to any fiber of $X$ to prove exactness, as we did in the proof of Proposition \ref{prop_complex_5}. More precisely we will show that, if $[U]\in X$, the induced complex of vector spaces
\begin{equation*} 
(\cU^\vee(-1)\otimes S^+)|_{[U]} \to (\wedge^3 V\otimes \cO)|_{[U]} \to T|_{[U]}
\end{equation*}
is exact. From this it will follow that the complex $$ \cU^\vee(-1)\otimes S^+ \to \wedge^3 V\otimes \cO \to T\to \cU_{\omega_1+\omega_2} \to 0 $$ is exact. The result will then follow by noticing that, since $\ss(\cU^\vee(-1)\otimes S^+)=\cU^\vee(-2)\oplus \cU_{\omega_1+\omega_2}(-2) \oplus \wedge^3 \cU^\vee(-2) \oplus \cU_{\omega_1+\omega_4}(-2)\oplus \cU \oplus \cU^\vee$ and $\ss(\wedge^3 V)=\wedge^3 \cU \oplus \cU_{\omega_1+\omega_4}(-2)\oplus \cU \oplus \cU_{\omega_2+\omega_5}(-2)\oplus \cU^\vee \oplus \wedge^3 \cU^\vee$, the semisimple reduction of the kernel of $\wedge^3 \eta \circ (i\otimes \id)$ is necessarily equal to $\cU^\vee(-2) \oplus \cU_{\omega_1+\omega_2}(-2)$.
\medskip

Let $[U]\in X$ be any point. Then $(\cU^\vee(-1))|_{[U]}\simeq U^\vee$ is a subspace of $(S^-\otimes \cO)|_{[U]}=S^-=\bigoplus_i \wedge^{2i+1}U^\vee$ as $L(\rP_6)$-representations. Letting $$t:=(\wedge^3 \eta \circ (i\otimes \id))_{[U]}: U^\vee \otimes \wedge^3 V^\vee \to S^+=\bigoplus_i \wedge^{2i}U^\vee $$ and $$u:=(\xi\circ (i\otimes \id))_{[U]}:U^\vee \otimes S^+ \to \wedge^3 V$$ and applying Lemma \ref{lem_lin_algebra}, we deduce the following:
$$\Imm(u)^\perp=\{ v\in \wedge^3 V^\vee \mid \forall f\in U^\vee \subset S^-=\bigoplus_i \wedge^{2i+1}U^\vee,\mbox{ } t(f\otimes v)=0 \}\subset \wedge^3 V^\vee.$$
Since we can treat everything as $L(\rP_6)$-equivariant/homogeneous, we can decompose $\wedge^3 V=\wedge^3 U \oplus U^{\omega_1+\omega_4}\oplus U \oplus U^{\omega_2+\omega_5}\oplus U^\vee \oplus \wedge^3 U^\vee$ (here, by abuse of notation, we denoted by $U^\omega$ the $\SL(U)$-representation with highest weight $\omega$). By $L(\rP_6)$-equivariance, each of these factors is either completely contained in $\Imm(u)^\perp$ or it does not intersect non-trivially $\Imm(u)^\perp$. In order to distinguish the two cases it is thus sufficient to decide whether a non-zero vector in a given factor belongs to $\Imm(u)^\perp$ or not. We thus have six cases to deal with. We will denote by $u_1,\dots,u_6$ a basis of $U$ and by $w_1,\dots,w_6$ the dual basis. We will denote by $u_{ijk}=u_i\wedge u_j\wedge u_k$, $u_{ij}=u_i\wedge u_j$, $w_{ij}=w_i\wedge w_j$ and $w_{ijk}=w_i\wedge w_j\wedge w_k$; $\delta_{i,j}$ will denote Kronecker's delta.
\begin{itemize}
    \item[$\wedge^3 U$:] Let $0\neq u_{ijk}\in \wedge^3 U$ and $w_h\in U^\vee$. Then $t(w_h\otimes u_{ijk})=u_{ij}\lrcorner(u_k\lrcorner w_k)-u_{ik}\lrcorner(u_j\lrcorner w_h)+u_{ij}\lrcorner(u_k\lrcorner w_h)=0,$ for any $k=1,\dots,6$, so $\wedge^3 U\subset \Imm(u)^\perp$.
    \item[$U^{\omega_1+\omega_4}$:] Let $0\neq u_{ij}\wedge w_k\in U^{\omega_1+\omega_4}$ for $i \neq k$ and $j\neq k$, and $w_h\in U^\vee$. Then $t(w_h\otimes u_{ij}\wedge w_k)=u_{ij}\lrcorner (w_{kh})- (u_i\lrcorner w_k) (u_j\lrcorner w_h)+(u_j\lrcorner w_k)(u_i\lrcorner w_h)= 0$, so $U^{\omega_1+\omega_4} \subset \Imm(u)^\perp$.
    \item[$U$:] Let $0\neq \sum_i u_{ij}\wedge w_j\in U\subset \wedge^3 V$ and $w_h\in U^\vee$. Then $t(w_h\otimes (\sum_iu_{ij}\wedge w_j))=\sum_i(u_{ij}\lrcorner w_{jh}-(u_i\lrcorner w_j)(u_j\lrcorner w_h) +(u_j\lrcorner w_j)(u_i\lrcorner w_h))=\sum_i(-2\delta_{i,h} +\delta_{i,h})\neq 0$, so $U\cap \Imm(u)^\perp=0$.
    \item[$U^\vee$:] Let $0\neq \sum_i u_{i}\wedge w_{ij}\in U^\vee\subset \wedge^3 V$ and $w_h\in U^\vee$. Then $t(w_h\otimes (\sum_iu_{i}\wedge w_{ij}))=\sum_i(u_{i}\lrcorner w_{ijh}-w_i\wedge(u_i\lrcorner w_{jh}) +w_{j}\wedge (u_i\lrcorner w_{ih}))=\sum_i(2w_{jh}+\delta_{i,h}w_{ij})=\sum_i w_{jh}\neq 0$, so $U^\vee\cap \Imm(u)^\perp=0$.
    \item[$U^{\omega_2+\omega_5}$:] Let $0\neq u_{i}\wedge w_{jk}\in U^{\omega_2+\omega_5}$ for $i \neq j$ and $i\neq k$, and $w_h\in U^\vee$. Then $t(w_h\otimes u_{i}\wedge w_{jk})=u_{i}\lrcorner (w_{jkh})- w_j\wedge (u_i\lrcorner w_{kh})+w_{jk}(u_i\lrcorner w_h)=3\delta_{i,h}w_{jk}\neq 0$, so $U^{\omega_2+\omega_5} \cap \Imm(u)^\perp=0$.
    \item[$\wedge^3 U^\vee$:] Let $0\neq w_{ijk}\in \wedge^3 U^\vee$ and $w_h\in U^\vee$. Then $t(w_h\otimes w_{ijk})=w_{ijkh}\neq 0$, so $\wedge^3 U^\vee \cap \Imm(u)^\perp=0$.
\end{itemize}
The previous computations imply that $\Imm(u)^\perp=\wedge^3 U\oplus U^{\omega_4+\omega_1}\subset \wedge^3 V \simeq \wedge^3 V^\vee$. This is equivalent to the fact that $\Imm(u)$ is the kernel of $\wedge^3 V \to (T)|_{[U]}$. The statement of the proposition follows.
\end{proof}

\subsubsection{Complete orthogonality}

Here we show the following result. Set:
\[
Q'=\rL_{\langle \cB \rangle}(Q).
\]
\begin{proposition}
The exceptional bundles $Q$ and $Q'(1)$ are completely orthogonal.
\end{proposition}

\begin{proof}
We know that $Q'(1)$ is an exceptional object that $\Ext_X^\bullet(Q'(1),Q)=0$, so we have to check that $Q'$ is concentrated in degree 0 and that $\Ext_X^\bullet(Q,Q'(1))=0$.
First we check that:
\begin{equation} \label{HomPQ}
\Ext^\bullet_X(P,Q)= \Hom_X(P,Q)=V^{\omega_1}.
\end{equation}
To see this, recall \eqref{defQ} and use that $(\cU_{\omega_1},P)$ is exceptional to get, for all $p \ge 0$:
\[
\Ext^p_X(P,Q) \simeq \Ext^p_X(P,\cU_{\omega_1+\omega_2}).
\]
Next, apply $\Hom_X(-,\cU_{\omega_1+\omega_2})$ to the sequence \eqref{defP} defining $P$
and work as in Lemma \ref{lemma:canonical-extensionQ} to show that :
\begin{align*}
&\Ext^\bullet_X(\cO_X,\cU_{\omega_1+\omega_2}) = \HH^0(\cU_{\omega_1+\omega2}) \simeq V^{\omega_1+\omega_2}, \\
&\Ext^{>0}_X(\cU_{\omega_2},\cU_{\omega_1+\omega_2}) = 
 \Ext^1_X(\cU_{\omega_2},\cU_{\omega_1+\omega_2}) \simeq \HH^1(\cU_{\omega_1+\omega_2+\omega_4}(-2)) \simeq V^{\omega_1+\omega_2},\\
 & \Hom_X(\cU_{\omega_2},\cU_{\omega_1+\omega_2}) = \HH^0(\cU_{\omega_1}) \simeq V^{\omega_1}. 
\end{align*}
Hence \eqref{HomPQ} holds if and only if the cup-product map below is non-degenerate:
\[
\Ext^1_X(\cU_{\omega_2},\cO_X) \otimes \Hom_X(\cO_X,\cU_{\omega_1+\omega_2}) 
\to \Ext^1_X(\cU_{\omega_2},\cU_{\omega_1+\omega_2}) 
\]
However, by the above analysis, using the notation of the proof of Lemma \ref{lemma:canonical-extensionQ}, this map is the cup-product 
\[
\HH^1(\cL_{\omega_4-2\omega_6}) \otimes \HH^0(\cL_{\omega_1+\omega_2}) 
\to \HH^1(\cL_{\omega_1+\omega_2+\omega_4-2\omega_6}) 
\]
and therefore it is non-degenerate by \cite{dimitrov-roth}. So \eqref{HomPQ} is proved. The resulting evaluation map $V^{\omega_1} \otimes P \to Q$ is surjective, as it results by tensoring \eqref{defP} by $V^{\omega_1}$ and considering the evaluation map to \eqref{defQ}; more precisely $\rL_P (Q)$ is an exceptional homogeneous bundle fitting into:
\begin{equation} \label{N1 and K}
0 \to \cU_{\omega_1}^\vee \to \rL_P (Q) \to N \to 0, \qquad \text{with} \qquad 0 \to \cU_{\omega_1} \xrightarrow{\alpha} N \to K \to 0,
\end{equation}
where $K$ is the kernel of the map $T \to \cU_{\omega_1+\omega_2}$ of Proposition \ref{prop_complex_6}. Note that $N$ fits into:
\begin{equation} \label{N1}
0 \to N \to V^{\omega_1} \otimes \cU_{\omega_2} \to \cU_{\omega_1+\omega_2} \to 0.
\end{equation}

Next, we show that $\cU_{\omega_1}$ is completely orthogonal to $\rL_P(Q)$, so we need to prove
\begin{equation} \label{UperpL}
    \Ext^\bullet_X(\cU_{\omega_1},\rL_P(Q))=0.
\end{equation} 
Using \eqref{N1 and K} and \eqref{N1}, one checks that \eqref{UperpL} is proved once we show that $\alpha$ induces a non-zero map:
\begin{equation} \label{mustbe}
\Ext^1_X(N,\cU_{\omega_1}^\vee)
\to \Ext^1_X(\cU_{\omega_1},\cU_{\omega_1}^\vee)
\end{equation}

To achieve this, first note that, as a consequence of the definition of $K$ in terms of $T$ and the definition of $T$, $K$ is an extension $$
    0\to \cU_{\omega_2+\omega_5}(-1)\to K \to \cU_{\omega_3}\to 0.
$$

Next, we work as in the proof of Lemma \ref{lemma:canonical-extensions} to check:
\begin{align*}
&\Ext_X^\bullet(\cU_{\omega_2+\omega_5}(-1),\cU_{\omega_1}^\vee)=
\Ext_X^1(\cU_{\omega_2+\omega_5}(-1),\cU_{\omega_1}^\vee) \simeq \HH^1(\cU_{\omega_4}(-2)) \simeq \HH^1(\Omega_X) \simeq \bC, \\
&\Ext_X^\bullet(\cU_{\omega_3},\cU_{\omega_1}^\vee)=\Ext_X^2(\cU_{\omega_3},\cU_{\omega_1}^\vee)\simeq \HH^2(\cU_{\omega_3+\omega_5}(-3)) \simeq \HH^2(\Omega_X^2) \simeq \bC, \\ 
&\Ext_X^1(\cU_{\omega_3},\cU_{\omega_2+\omega_5}(-1)) \simeq \HH^1(\cU_{\omega_4}(-2)) \simeq \HH^1(\Omega_X) \simeq \bC.
\end{align*}

Then, $\Ext^\bullet_X(K,\cU_{\omega_1}^\vee)$ vanishes if and only the following cup-product map is non-degenerate:
\[
\Ext_X^1(\cU_{\omega_3},\cU_{\omega_2+\omega_5}(-1)) \otimes \Ext_X^1(\cU_{\omega_2+\omega_5}(-1),\cU_{\omega_1}^\vee) \to 
\Ext_X^2(\cU_{\omega_3},\cU_{\omega_1}^\vee)
\]

But from we have just seen this map is the cup-product in cohomology:
\[
\HH^1(\Omega_X) \otimes \HH^1(\Omega_X) \to \HH^2(\Omega^2_X),
\]
and therefore it is non-degenerate. This proves $\Ext^\bullet_X(K,\cU_{\omega_1}^\vee)=0$.

Now we can check that \eqref{mustbe} is non-zero. Indeed, assume it was. Then $\alpha$ induces an exact sequence :
\[
0 \to \cU_{\omega_1} \oplus \cU_{\omega_1}^\vee \to \rL_P(Q) \to K \to 0.
\]
But then, since $\Ext^1_X(K,\cU_{\omega_1}^\vee)=0$, $\cU_{\omega_1}^\vee$ is a direct summand of $\rL_P(Q)$, which cannot happen since $\rL_P(Q)$ is exceptional.

We have now proved \eqref{UperpL}. Moreover, we get that $\alpha$ induces a diagram:
\[
\xymatrix@-2ex{
& 0 \ar[d]& 0 \ar[d] \\
& \cU_{\omega_1}^\vee \ar@{=}[r] \ar[d] & \cU_{\omega_1}^\vee \ar[d] \\
0 \ar[r] & V^{\omega_1} \otimes \cO_X \ar[d]\ar[r] & \rL_P(Q) \ar[r] \ar[d]&  K \ar[r] \ar@{=}[d] & 0 \\
0 \ar[r] & \cU_{\omega_1} \ar^-\alpha[r] \ar[d]& N \ar[r] \ar[d]& K \ar[r] \ar[d]& 0 \\
& 0 & 0 & 0
}
\]
The leftmost column is the tautological sequence \eqref{eq:tautological-exact-sequence} because the cup-product above is non-degenerate and thus \eqref{UperpL} is proved.
Moreover we get, from the previous diagram:
\[
Q' = \rL_{\langle \cB \rangle}(Q) \simeq \rL_{\cO_X}(\rL_P(Q)) \simeq \rL_{\cO_X}(K).
\]
Finally using the definition of $K$ we check $\HH^\bullet(K)=\HH^0(K)=V^{\omega_3}$ hence, by Proposition \ref{prop_complex_6}, $Q'$ is concentrated in degree 0 and we obtain:
\[
Q'(1) \in \langle \cU_{\omega_1}(-1),\cU_{\omega_1+\omega_2}(-1),\cU_{\omega_1}\rangle.
\]
Therefore, using Lemma \ref{it's exceptional} and \eqref{defQ} we get $\Ext_X^\bullet(Q,Q'(1))=0$.
\end{proof}

\begin{remark}
\label{rem_R6}
One can actually prove that the bundle $R_6$ appearing in Proposition \ref{prop_complex_6} satisfies:
\[
R_6 \simeq Q(-2) \simeq \rL_{\cU_{\omega_1}(-1)}(\rL_{\langle \cB \rangle}(Q)).
\]
Indeed, we checked that $K \simeq \rL_{\langle \cB \rangle}(Q)$ and one can prove $\Ext^\bullet_X(\cU_{\omega_1}(-1),K)\simeq V_{\omega_5}$, so $R_6 \simeq \rL_{\cU_{\omega_1}(-1)}(K)$ is exceptional, hence indecomposable, so by  Proposition \ref{prop_complex_6} it must be isomorphic to $Q(-2)$.
\end{remark}

\section{Generating more objects}
\label{more}

We come back to $X=\Spin_{12}/\rP_6$. The goal of this section is to show that the exceptional full triangulated subcategory $\cD$ of $\Db(X)$ generated by the exceptional Lefschetz collection of Section \ref{warming up} contains a bunch of vector bundles, which will be needed in the proof 
that $\cD^\perp=0$. From \eqref{eq:collection-d6-p6} we immediately have
\begin{align}
  \label{eq:containment-O}
   \cO(t) & \in \cD \quad \text{for} \quad t \in [0,9], \\
  \label{eq:containment-Udual}
   \cU^\vee(t) & \in \cD \quad \text{for} \quad t \in [0,9], \\
  \label{eq:containment-Lambda^2Udual}
   \wedge^2 \cU^\vee(t) & \in \cD \quad \text{for} \quad t \in [0,9], \\
  \label{eq:containment-Sigma21-Udual-initial}
   \Sigma^{2,1} \cU^\vee(t) & \in \cD \quad \text{for} \quad t \in [0,1].
\end{align}

Often we are going to use the tautological exact sequence
\begin{equation}\label{eq:tautological-exact-sequence}
  0 \to \cU \to V \otimes \cO \to \cU^{\vee} \to 0.
\end{equation}

\bigskip

Twisting \eqref{eq:tautological-exact-sequence} by $\cO(t)$ with $t \in [0,9]$ and
using \eqref{eq:containment-O}, \eqref{eq:containment-Udual} we immediately obtain
\begin{equation}\label{eq:containment-U}
  \cU(t) \in \cD \quad \text{for} \quad t \in [0,9].
\end{equation}

\bigskip

We also note that for $j \in [0,6]$ we have isomorphisms
\begin{equation*}
  \wedge^j \cU^\vee \simeq \wedge^{6-j} \cU (2)
  \quad \text{and} \quad
  \wedge^j \cU \simeq \wedge^{6-j} \cU^\vee (-2).
\end{equation*}

\bigskip

\begin{lemma}\label{lemma:filtrations-spinor-representations}
Seeing the spinor representations $\rV^{\omega_5}$ and $\rV^{\omega_6}$ as vector bundles on $X$, we have:
  \begin{enumerate}[label=\roman*)]
    \item the vector bundle $\rV^{\omega_5} \otimes \cO$ has an increasing filtration,
    whose factors are of the form
    \begin{equation*}
      \wedge^{2i+1} \, \cU^{\vee}(-1) \quad \text{for} \quad t \in [0,2]
    \end{equation*}
    \item the vector bundle $\rV^{\omega_6} \otimes \cO$ has an increasing filtration,
    whose factors are of the form
    \begin{equation*}
      \wedge^{2i} \, \cU^{\vee}(-1) \quad \text{for} \quad t \in [0,3]
    \end{equation*}
  \end{enumerate}
\end{lemma}

\begin{proof}
  This follows from \cite[Proposition 6.3]{Ku08a} (and is the same filtration $\cF_\bullet$ described in Section \ref{sec_spinor_bundles}).
\end{proof}

As a corollary we obtain the following.
\begin{corollary} 
  We have
  \begin{equation}\label{eq:containment-exterior-powers-Udual}
    \wedge^j \cU^\vee(t) \in \cD \quad \text{for} \quad
    \begin{cases}
      t \in [0,9] \quad \text{if} \quad j \in [0,2], \\
      t \in [0,7] \quad \text{if} \quad j \in [3,4], \\
      t \in [-2,7] \quad \text{if} \quad j = 5, \\
      t \in [-2,7] \quad \text{if} \quad j = 6.
    \end{cases}
  \end{equation}
\end{corollary}

\begin{proof}
  Cases with $j \in [0,2]$ we have already considered. We treat each $j \in [3,6]$ separately.
  \begin{enumerate}
    \item \textit{Case $j = 3$.} Twisting $\rV^{\omega_5} \otimes \cO$ by $\cO(t)$
    with $t \in [1,8]$, using Lemma \ref{lemma:filtrations-spinor-representations},
    the isomorphism $\cU(1) = \wedge^5 \cU^\vee(-1)$,
    \eqref{eq:containment-O},\eqref{eq:containment-Udual},\eqref{eq:containment-U},
    we obtain the claim.

    \smallskip

    \item \textit{Case $j = 4$.} Twisting $\rV^{\omega_6} \otimes \cO$ by $\cO(t)$
    with $t \in [1,8]$, using Lemma \ref{lemma:filtrations-spinor-representations},
    \eqref{eq:containment-O},\eqref{eq:containment-Lambda^2Udual}, we obtain the claim.

    \smallskip

    \item \textit{Case $j = 5$.} Twisting the isomorphism $\cU(2) = \wedge^5 \cU^\vee$
    by $\cO(t)$ with $t \in [-2,7]$ and using \eqref{eq:containment-U}, we obtain the claim.

    \smallskip

    \item \textit{Case $j = 6$.} Since $\wedge^6 \cU^\vee = \det(\cU^\vee) = \cO(2)$,
    the claim follows from \eqref{eq:containment-O}.
  \end{enumerate}
\end{proof}

\begin{lemma} 
  We have
  \begin{equation}\label{eq:containment-S^nUdual}
    S^j \cU^\vee(t) \in \cD \quad \text{for} \quad
    \begin{cases}
      t \in [0,9] \quad \text{if} \quad j \in [0,1], \\
      t \in [2,9] \quad \text{if} \quad j \geq 2.
    \end{cases}
  \end{equation}
\end{lemma}

\begin{proof}
  For $j \in [0,1]$ the statement are known by \eqref{eq:containment-O} and
  \eqref{eq:containment-Udual}.

  \medskip

  \noindent \textit{Case $j = 2$.} From \eqref{eq:tautological-exact-sequence}
  we obtain the exact sequence
  \begin{equation*}
    0 \to \wedge^2 \cU \to \wedge^2 V \otimes \cO \to V \otimes \cU^\vee \to S^2 \cU^\vee \to 0.
  \end{equation*}
  Twisting this sequence by $\cO(t)$ with $t \in [2,9]$, using \eqref{eq:containment-O},
  \eqref{eq:containment-Udual}, the isomorphism $\wedge^2 \cU = \wedge^4 \cU^\vee(-2)$, \eqref{eq:containment-exterior-powers-Udual}, we see that all the terms of the
  sequence except for $S^2 \cU^\vee(t)$ are contained in $\cD$. Hence, the same holds
  for $S^2 \cU^\vee(t)$.

  \medskip

  \noindent \textit{Cases $j \geq 3$.}
  We argue by induction. For each $j \geq 3$ we consider the exact sequence
  \begin{equation*}
    0 \to \wedge^j \cU \to \wedge^j V \otimes \cO \to \wedge^{j-1} V \otimes \cU^\vee \to \wedge^{j-2} V \otimes S^2 \cU^\vee
    \to \dots \to V \otimes S^{j-1} \cU^\vee \to S^j \cU^\vee.
  \end{equation*}
  All the middle terms twisted by
  $\cO(t)$ with $t \in [2,9]$ are contained in $\cD$ by the induction assumption.
  For $j \in [3,6]$ the term $\wedge^j \cU(t) = \wedge^{6-j} \cU^\vee(t-2)$ is also in $\cD$
  for $t \in [2,9]$ by \eqref{eq:containment-exterior-powers-Udual}.
  For $j \geq 7$ this term vanishes. Hence, the claim follows.
\end{proof}

\begin{lemma}
  We have
  \begin{align}\label{eq:containment-S^2U}
    S^2 \cU(t) & \in   \cD &&\text{for} \quad t \in [0,9], \\
    \label{eq:containment-U-tensor-Udual}
    \cU \otimes \cU^\vee(t) & \in \cD && \text{for} \quad t \in [2,9],\\
\label{eq:containment-Udual-tensor-Udual}
    \cU^\vee \otimes \cU^\vee(t) & \in \cD &&\text{for} \quad t \in [2,9].
  \end{align}
\end{lemma}

\begin{proof}
  From \eqref{eq:tautological-exact-sequence} we get an exact sequence
  \begin{equation*}
    0 \to S^2 \cU \to S^2 V \otimes \cO \to V \otimes \cU^\vee \to \wedge^2 \cU^{\vee} \to 0.
  \end{equation*}
  Twisting this sequence by $\cO(t)$ with $t \in [0,9]$ and using \eqref{eq:containment-O},
  \eqref{eq:containment-Udual}, \eqref{eq:containment-exterior-powers-Udual}
  we obtain \eqref{eq:containment-S^2U}.

  \medskip

  One can reformulate \eqref{eq:tautological-exact-sequence} by saying that the bundle
  $V \otimes \cO$ has a filtration with factors $\cU$ and $\cU^\vee$. Then, taking
  symmetric square, we obtain on  $S^2 V \otimes \cO$ a filtration with factors $
    S^2 \cU , \cU \otimes \cU^\vee ,  S^2 \cU^\vee$. Twisting it by $\cO(t)$ with $t \in [2,9]$, using \eqref{eq:containment-O}, \eqref{eq:containment-S^2U}, and \eqref{eq:containment-S^nUdual}, we get
  \eqref{eq:containment-U-tensor-Udual}.

  \medskip

  Finally, tensoring \eqref{eq:tautological-exact-sequence} with $\cU^\vee$ we get $0 \to \cU \otimes \cU^\vee \to V \otimes \cU^\vee \to \cU^\vee \otimes \cU^\vee \to 0$. Together with \eqref{eq:containment-Udual} and \eqref{eq:containment-U-tensor-Udual}
  it implies \eqref{eq:containment-Udual-tensor-Udual}.
\end{proof}

\bigskip

Recall that from Proposition \ref{prop_complex_6} and Remark \ref{rem_R6} we have the exact sequence
\begin{equation*}
0\to Q(-2) \to \cU^\vee(-1)\otimes S^+ \to \wedge^3 V\otimes \cO \to T \to \Sigma^{2,1}\cU^\vee \to 0 ,
\end{equation*}
with 
$T$ defined by $0\to \cU^\vee \to V\otimes \wedge^2 \cU^\vee \to T \to 0$.
Twisting this sequence by $\cO(2)$ and
using \eqref{eq:containment-O}-\eqref{eq:containment-Sigma21-Udual-initial},
we obtain $\Sigma^{2,1} \cU^\vee(2) \in \cD$. Iterating this process one shows
\begin{equation}\label{eq:containment-Sigma21-Udual}
  \Sigma^{2,1} \cU^\vee \in \cD \quad \text{for} \quad t \in [0,9].
\end{equation}

\begin{lemma}
  We have
  \begin{align}\label{eq:containment-Lambda^2U-tensor-Udual}
    &\wedge^2 \cU \otimes \cU^\vee(t) \in \cD && \text{for} \quad t \in [2,9],\\
    &\label{eq:containment-Lambda^2U-tensor-U}
    \wedge^2 \cU \otimes \cU(t) \in \cD && \text{for} \quad t \in [2,9], \\
&\label{eq:containment-Udual-tensor-Lambda^2Udual}
        \cU^\vee \otimes \wedge^2 \cU^\vee(t) \in \cD && \text{for} \quad t \in [0,7],\\
  & \label{eq:containment-U-tensor-Lambda^2Udual}
    \cU \otimes \wedge^2 \cU^\vee(t) \in \cD && \text{for} \quad t \in [0,7].
  \end{align}
\end{lemma}

\begin{proof}
  To show \eqref{eq:containment-Lambda^2U-tensor-Udual} we consider the exact sequence
  \begin{equation*}
    0 \to \wedge^2 \cU \to \wedge^2 V \otimes \cO \to V \otimes \cU^\vee \to S^2 \cU^\vee \to 0
  \end{equation*}
  obtained from \eqref{eq:tautological-exact-sequence}. Tensoring it by $\cU^\vee$
  we obtain the exact sequence
  \begin{equation*}
    0 \to \wedge^2 \cU \otimes \cU^\vee \to \wedge^2 V \otimes \cU^\vee \to V \otimes \cU^\vee \otimes \cU^\vee \to S^2 \cU^\vee \otimes \cU^\vee \to 0.
  \end{equation*}
  Note that we have
  \begin{equation*}
    \begin{aligned}
      & \cU^\vee \otimes \cU^\vee = \wedge^2 \cU^\vee \oplus S^2 \cU^\vee, \\
      & S^2 \cU^\vee \otimes \cU^\vee = S^3 \cU^\vee \oplus \Sigma^{2,1} \cU^\vee.
    \end{aligned}
  \end{equation*}
  Twisting by $\cO(t)$ with $t \in [2,9]$ and using
  \eqref{eq:containment-exterior-powers-Udual},
  \eqref{eq:containment-S^nUdual}, \eqref{eq:containment-Sigma21-Udual}
  we obtain the claim.

  \bigskip

  To show \eqref{eq:containment-Lambda^2U-tensor-U} we tensor the exact sequence
  \eqref{eq:tautological-exact-sequence} by $\wedge^2 \cU$ and get the exact sequence
  \begin{equation*}
    0 \to \cU \otimes \wedge^2 \cU \to V \otimes \wedge^2 \cU \to \cU^\vee \otimes \wedge^2 \cU \to 0.
  \end{equation*}
  Twisting by $\cO(t)$ with t $\in [2,9]$,
  using the isomorphism $\wedge^2 \cU = \wedge^4 \cU^\vee(-2)$,
  \eqref{eq:containment-exterior-powers-Udual},
  \eqref{eq:containment-Lambda^2U-tensor-Udual}
  we get the claim.

  \bigskip

  To show \eqref{eq:containment-Udual-tensor-Lambda^2Udual} we note
  \begin{equation*}
    \cU^\vee \otimes \wedge^2 \cU^\vee \simeq \wedge^3 \cU^\vee \oplus \Sigma^{2,1} \cU^\vee.
  \end{equation*}
  Twisting by $\cO(t)$ with $t \in [0,7]$, using \eqref{eq:containment-Sigma21-Udual}
  and \eqref{eq:containment-exterior-powers-Udual} we obtain the claim.

  \bigskip

  To show \eqref{eq:containment-U-tensor-Lambda^2Udual} we tensor the exact sequence
  \eqref{eq:tautological-exact-sequence} by $\wedge^2 \cU^\vee$ to get
  \begin{equation*}
    0 \to \cU \otimes \wedge^2 \cU^\vee \to V \otimes \wedge^2 \cU^\vee \to
    \cU^\vee \otimes \wedge^2 \cU^\vee \to 0.
  \end{equation*}
  Twisting by $\cO(t)$ with $t \in [0,7]$, using \eqref{eq:containment-Udual-tensor-Lambda^2Udual}
  and \eqref{eq:containment-exterior-powers-Udual} we obtain the claim.
\end{proof}

At this point we have proved the following.
\begin{corollary}\label{corollary:almost-all-containments} We have:
\begin{enumerate}[label=\roman*)]
    \item \label{Ci} $\cU^\vee \otimes \wedge^0 \cU^\vee(t) \in \cD$ for $t \in [0,9]$,
    \item \label{Cii} $\cU^\vee \otimes \wedge^1 \cU^\vee(t) \in \cD$ for $t \in [2,9]$,
    \item \label{Ciii} $\cU^\vee \otimes \wedge^2 \cU^\vee(t) \in \cD$ for $t \in [0,7]$, 
    \item \label{Civ} $\cU^\vee \otimes \wedge^4 \cU^\vee(t) \in \cD$ for $t \in [0,7]$,
    \item \label{Cv} $\cU^\vee \otimes \wedge^5 \cU^\vee(t) \in \cD$ for $t \in [0,7]$,
    \item  \label{Cvi} $\cU^\vee \otimes \wedge^6 \cU^\vee(t) \in \cD$  for $t \in [-2,7]$.
\end{enumerate}  
\end{corollary}

\begin{proof}
We already proved these statements. Indeed, \ref{Ci} is \eqref{eq:containment-Udual}, \ref{Cii} is
    \eqref{eq:containment-Udual-tensor-Udual}, 
    \ref{Ciii} is \eqref{eq:containment-Udual-tensor-Lambda^2Udual},
    \ref{Civ} follows from $\cU^\vee \otimes \wedge^4 \cU^\vee(t) \simeq \cU^\vee \otimes \wedge^2\cU(t+2)$ and \eqref{eq:containment-Lambda^2U-tensor-Udual},
    \ref{Cv} follows from $\cU^\vee \otimes \wedge^5 \cU^\vee(t) \simeq \cU^\vee \otimes \cU(t+2)$ and \eqref{eq:containment-U-tensor-Udual}
and \ref{Cvi} follows from $\wedge^6 \cU^\vee \simeq \cO(2)$.
\end{proof}

Thus, we are still missing the objects $\cU^\vee \otimes \wedge^3 \cU^\vee(t)$,
and the range of $t$ for $\cU^\vee \otimes \cU^\vee(t)$ needs to be extended.
This is our next goal.

\begin{lemma} We have
  \begin{align}\label{eq:containment-Udual-tensor-Lambda^3Udual}
    \cU^\vee \otimes \wedge^3 \cU^\vee(t) &\in \cD & \text{for} \quad t \in [2,7],\\
  \label{eq:containment-U-tensor-Lambda^3Udual}
    \cU \otimes \wedge^3 \cU^\vee(t) &\in \cD & \text{for} \quad t \in [2,7] ,\\
\label{eq:containment-Udual-tensor-Lambda^3U}
    \cU^\vee \otimes \wedge^3 \cU(t) &\in \cD &\text{for} \quad t \in [4,9].
  \end{align}
\end{lemma}

\begin{proof}
  To show \eqref{eq:containment-Udual-tensor-Lambda^3Udual} we proceed as follows.
  By Lemma \ref{lemma:filtrations-spinor-representations} we have
  \begin{equation*}
    V^{\omega_5} \otimes \cO = [\cU^{\vee}(-1), \, \wedge^3\cU^{\vee}(-1), \, \wedge^5\cU^{\vee}(-1) = \cU(1)].
  \end{equation*}
  Tensoring it by $\cU^\vee$ we get
  \begin{equation*}
    V^{\omega_5} \otimes \cU^\vee = [\cU^{\vee} \otimes \cU^\vee (-1), \,
    \wedge^3\cU^{\vee} \otimes \cU^\vee (-1), \, \cU \otimes \cU^\vee (1)].
  \end{equation*}
  Twisting it by $\cO(t)$ with $t \in [3,8]$ and using
  \eqref{eq:containment-Udual},
  \eqref{eq:containment-U-tensor-Udual},
  \eqref{eq:containment-Udual-tensor-Udual}
  we obtain the claim.

  \bigskip

  To show \eqref{eq:containment-U-tensor-Lambda^3Udual} one can tensor the exact sequence
  \eqref{eq:tautological-exact-sequence} by $\wedge^3\cU^{\vee}$ to get
  \begin{equation*}
    0 \to \cU \otimes \wedge^3\cU^{\vee} \to V \otimes \wedge^3\cU^{\vee} \to
    \cU^\vee \otimes \wedge^3\cU^{\vee} \to 0.
  \end{equation*}
  Now we twist by $\cO(t)$ with $t \in [2,7]$ and use
  \eqref{eq:containment-Udual-tensor-Lambda^3Udual} and
  \eqref{eq:containment-exterior-powers-Udual}.

  \bigskip

  To show \eqref{eq:containment-Udual-tensor-Lambda^3U} we use the inclusion
  \eqref{eq:containment-Udual-tensor-Lambda^3Udual} and the isomorphism
  $\wedge^3\cU^{\vee} = \wedge^3\cU(2)$.
\end{proof}

\begin{lemma}
  We have
  \begin{align}\label{eq:containment-Lambda^2U-tensor-Lambda^2Udual}
    \wedge^2\cU \otimes \wedge^2\cU^{\vee} (t) &\in \cD && \text{for} \quad t \in [4,7], \\
\label{eq:containment-Lambda^2Udual-tensor-Lambda^2Udual}
    \wedge^2\cU^{\vee} \otimes \wedge^2 \cU^\vee(t) &\in \cD && \text{for} \quad t \in [2,5].
  \end{align}
\end{lemma}

\begin{proof}
  To show \eqref{eq:containment-Lambda^2U-tensor-Lambda^2Udual} we consider the filtration
  \begin{equation*}
    \wedge^4 V \otimes \cO = [
    \wedge^2\cU^\vee(-2), \, \wedge^3\cU \otimes \cU^\vee, \,
    \wedge^2\cU \otimes \wedge^2\cU^\vee, \, \cU \otimes \wedge^3\cU^\vee, \,
    \wedge^4\cU^\vee ].
  \end{equation*}
  Twisting by $\cO(t)$ with $t \in [4,7]$ and using
  \eqref{eq:containment-O},
  \eqref{eq:containment-exterior-powers-Udual},
  \eqref{eq:containment-U-tensor-Lambda^3Udual},
  \eqref{eq:containment-Udual-tensor-Lambda^3U}
  we get the claim.

  \medskip

  To show \eqref{eq:containment-Lambda^2Udual-tensor-Lambda^2Udual} we consider
  the filtration of Lemma \ref{lemma:filtrations-spinor-representations}
  \begin{equation*}
    V^{\omega_6} \otimes \cO = [\cO(-1), \, \wedge^2\cU^{\vee}(-1), \,
    \wedge^2\cU(1), \, \cO(1)],
  \end{equation*}
  where we have used the isomorphisms $\wedge^4 \cU^\vee(-1) = \wedge^2 \cU(1)$
  and $\wedge^6 \cU^\vee(-1) = \cO(1)$.
  Tensoring it by $\wedge^2 \cU^\vee$ we get
  \begin{equation*}
    V^{\omega_6} \otimes \wedge^2 \cU^\vee = [
    \wedge^2 \cU^\vee (-1), \, \wedge^2\cU^{\vee} \otimes \wedge^2 \cU^\vee (-1), \,
    \wedge^2\cU \otimes \wedge^2 \cU^\vee (1), \, \wedge^2 \cU^\vee (1)].
  \end{equation*}
  Twisting it by $\cO(t)$ with $t \in [3,6]$ and using
  \eqref{eq:containment-exterior-powers-Udual},
  \eqref{eq:containment-Lambda^2U-tensor-Lambda^2Udual}
  we obtain the claim.
\end{proof}

\begin{lemma}
  We have
  %
  %
  \begin{align}
    \label{eq:containment-Sigma31-Udual}
    & \Sigma^{3,1} \cU^\vee(t)  \in \cD && \text{for} \quad t \in [2,9], \\
    \label{eq:containment-Sigma211-Udual}
    & \Sigma^{2,1,1} \cU^\vee(t)  \in \cD && \text{for} \quad t \in [2,7], \\
    \label{eq:containment-Sigma22-Udual}
    & \Sigma^{2,2} \cU^\vee(t)  \in \cD && \text{for} \quad t \in [2,7].
  \end{align}
\end{lemma}

\begin{proof}
  Our first step is to note that by the Littlewood-Richardson rule we have
  \begin{equation*}
    \cU^\vee \otimes \wedge^3 \cU^\vee = \Sigma^{2,1,1} \cU^\vee \oplus \wedge^4 \cU^\vee,
  \end{equation*}
  \begin{equation*}
    \wedge^2 \cU^\vee \otimes \wedge^2 \cU^\vee = \Sigma^{2,1,1} \cU^\vee \oplus
    \Sigma^{2,2} \cU^\vee \oplus \wedge^4 \cU^\vee.
  \end{equation*}
  Hence, by \eqref{eq:containment-exterior-powers-Udual} inclusions
  \eqref{eq:containment-Udual-tensor-Lambda^3Udual} and
  \eqref{eq:containment-Lambda^2Udual-tensor-Lambda^2Udual} immediately imply
  \begin{equation}\label{eq:prelim-containment-Sigma211-Udual}
    \Sigma^{2,1,1} \cU^\vee(t) \in \cD \quad \text{for} \quad t \in [2,7],
  \end{equation}
  \begin{equation}\label{eq:prelim-containment-Sigma22-Udual}
    \Sigma^{2,2} \cU^\vee(t) \in \cD \quad \text{for} \quad t \in [2,5].
  \end{equation}
  This proves \eqref{eq:containment-Sigma211-Udual}, but it is not quite enough
  to prove \eqref{eq:containment-Sigma22-Udual}.

  \bigskip

  Our second step is to deal with $\Sigma^{3,1} \cU^\vee$. Let us consider the exact
  sequence
  \begin{equation*}
    0 \to \wedge^3 \cU \to \wedge^3 V \otimes \cO \to \wedge^2 V \otimes \cU^\vee
    \to V \otimes S^2 \cU^\vee \to S^3 \cU^\vee \to 0
  \end{equation*}
  obtained from \eqref{eq:tautological-exact-sequence}. After twisting by
  $\cO(2)$ and using the isomorphism $\wedge^3 \cU = \wedge^3 \cU^\vee(-2)$
  we rewrite the above sequence as
  \begin{equation*}
    0 \to \wedge^3 \cU^\vee \to \wedge^3 V \otimes \cO(2) \to \wedge^2 V \otimes \cU^\vee(2)
    \to V \otimes S^2 \cU^\vee(2) \to S^3 \cU^\vee(2) \to 0.
  \end{equation*}
  Tensoring it by $\cU^\vee$ we obtain the exact sequence
  \begin{multline*}
  0 \to \wedge^3 \cU^\vee \otimes \cU^\vee \to \wedge^3 V \otimes \cU^\vee(2)
  \to \wedge^2 V \otimes \cU^\vee \otimes \cU^\vee(2) \to \\
  \to V \otimes S^2 \cU^\vee \otimes \cU^\vee(2) \to S^3 \cU^\vee \otimes \cU^\vee(2) \to 0.
  \end{multline*}
  Now we note that by the Littlewood-Richardson rule we have
  \begin{equation*}
    S^2 \cU^\vee \otimes \cU^\vee = S^3 \cU^\vee \oplus \Sigma^{2,1} \cU^\vee
    \quad \text{and} \quad
    S^3 \cU^\vee \otimes \cU^\vee = S^4 \cU^\vee \oplus \Sigma^{3,1} \cU^\vee
  \end{equation*}
  Therefore, tensoring the above sequence by $\cO(t)$ with $t \in [2,7]$ and using
  \eqref{eq:containment-Udual-tensor-Lambda^3Udual},
  \eqref{eq:containment-Udual},
  \eqref{eq:containment-Udual-tensor-Udual},
  \eqref{eq:containment-S^nUdual},
  \eqref{eq:containment-Sigma21-Udual},
  we conclude
  \begin{equation}\label{eq:prelim-containment-Sigma31-Udual}
    \Sigma^{3,1} \cU^\vee(t) \in \cD \quad \text{for} \quad t \in [4,9].
  \end{equation}
  It is not quite enough for \eqref{eq:containment-Sigma31-Udual}, but we are going
  to fix this soon.

  \bigskip

  As third step we proceed as follows.
  Recall again the exact sequence from Proposition \ref{prop_complex_6} and Remark \ref{rem_R6}
  \begin{equation}\label{eq:mutation-object-Q-recall}
   0\to Q(-2) \to \cU^\vee(-1)\otimes S^+ \to \wedge^3 V\otimes \cO \to T \to \Sigma^{2,1}\cU^\vee \to 0 .
  \end{equation}
  Using the Littlewood-Richardson rule and the definitions of $P$ and $Q$ we have
  \begin{equation*}
    \ss(P \otimes \cU^\vee) = \cU^\vee \oplus \wedge^3 \cU^\vee \oplus \Sigma^{2,1} \cU^\vee
  \end{equation*}
  and
  \begin{equation*}
    \ss(Q \otimes \cU^\vee) = \Sigma^{3,1} \cU^\vee \oplus \Sigma^{2,2} \cU^\vee
    \oplus \Sigma^{2,1,1} \cU^\vee \oplus S^2 \cU^\vee \oplus \wedge^2 \cU^\vee.
  \end{equation*}
  Hence, for $P \otimes \cU^\vee$ we have
  \begin{equation}\label{eq:prelim-containment-P-tensor-Udual}
    P \otimes \cU^\vee(t) \in \cD \quad \text{for} \quad t \in [0,7],
  \end{equation}
  as each individual factor is contained in $\cD$ by \eqref{eq:containment-Sigma21-Udual}
  and \eqref{eq:containment-exterior-powers-Udual}.

  In the same way we have
  \begin{equation}\label{eq:prelim-containment-Q-tensor-Udual}
    Q \otimes \cU^\vee(t) \in \cD \quad \text{for} \quad t \in [4,5],
  \end{equation}
  as each individual factor is contained in $\cD$ by
  \eqref{eq:containment-exterior-powers-Udual},
  \eqref{eq:containment-S^nUdual},
  \eqref{eq:prelim-containment-Sigma211-Udual},
  \eqref{eq:prelim-containment-Sigma22-Udual},
  \eqref{eq:prelim-containment-Sigma31-Udual}.

  \smallskip

  Tensoring \eqref{eq:mutation-object-Q-recall} by $\cU^\vee(t)$ with $t \in [6,7]$
  and using \eqref{eq:prelim-containment-P-tensor-Udual},
  \eqref{eq:prelim-containment-Q-tensor-Udual},
  \eqref{eq:containment-Udual},
  \eqref{eq:containment-Udual-tensor-Udual}, \eqref{eq:containment-Udual-tensor-Lambda^2Udual}
  we conclude that $Q \otimes \cU^\vee(t) \in \cD$ for $t \in [6,7]$.
  Similarly, tensoring \eqref{eq:mutation-object-Q-recall} by $\cU^\vee(t)$ with $t \in [4,5]$,
  we conclude $Q \otimes \cU^\vee(t) \in \cD$ for $t \in [2,3]$.
  Finally, tensoring \eqref{eq:mutation-object-Q-recall} by $\cU^\vee(3)$, we obtain
  $Q \otimes \cU^\vee(1) \in \cD$. Thus, we have shown $Q \otimes \cU^\vee(t) \in \cD \quad \text{for} \quad t \in [1,7]$. This inclusion, together with
  \eqref{eq:containment-S^nUdual},
  \eqref{eq:containment-exterior-powers-Udual},
  \eqref{eq:prelim-containment-Sigma211-Udual},
  \eqref{eq:prelim-containment-Sigma22-Udual},
  \eqref{eq:prelim-containment-Sigma31-Udual}
  allows to conclude first that $\Sigma^{2,2} \cU^\vee(t) \in \cD$ for $t \in [6,7]$,
  as all the other factors are already contained in $\cD$ with these twists.
  Then, similarly, we conclude that $\Sigma^{3,1} \cU^\vee(t) \in \cD$ for $t \in [2,3]$.
\end{proof}

\begin{corollary}
  We have
\begin{align}
    \label{eq:stronger-containment-Udual-tensor-Lambda^3Udual}
    &\wedge^3 \cU^\vee \otimes \cU^\vee(t) \in \cD && \text{for} \quad t \in [0,7],\\
  &\label{eq:stronger-containment-Sigma211-Udual}
    \Sigma^{2,1,1} \cU^\vee(t) \in \cD && \text{for} \quad t \in [0,7],\\
  &\label{eq:stronger-containment-Udual-tensor-Udual}
    \cU^\vee \otimes \cU^\vee(t) \in \cD && \text{for} \quad t \in [0,9].
\end{align}
\end{corollary}

\begin{proof}
  Let us consider again the exact sequence
  \begin{multline*}
  0 \to \wedge^3 \cU^\vee \otimes \cU^\vee \to \wedge^3 V \otimes \cU^\vee(2)
  \to \wedge^2 V \otimes \cU^\vee \otimes \cU^\vee(2) \to \\
  \to V \otimes S^2 \cU^\vee \otimes \cU^\vee(2) \to S^3 \cU^\vee \otimes \cU^\vee(2) \to 0.
  \end{multline*}
  as in the proof of the previous lemma. From the previous lemma,
  \eqref{eq:containment-exterior-powers-Udual},
  \eqref{eq:containment-S^nUdual},
  \eqref{eq:containment-Udual-tensor-Udual},
  \eqref{eq:containment-Sigma21-Udual},
  we know that all its terms except for $\wedge^3 \cU^\vee \otimes \cU^\vee$ are
  contained in $\cD$ with twists in $[0,7]$. Hence, the same holds for $\wedge^3 \cU^\vee \otimes \cU^\vee$.

  \medskip

  The inclusion \eqref{eq:stronger-containment-Sigma211-Udual} follows from
  \begin{equation*}
    \cU^\vee \otimes \wedge^3 \cU^\vee = \Sigma^{2,1,1} \cU^\vee \oplus \wedge^4 \cU^\vee,
  \end{equation*}
  combined with \eqref{eq:stronger-containment-Udual-tensor-Lambda^3Udual} and
  \eqref{eq:containment-exterior-powers-Udual}.

  \medskip

  Finally, we show \eqref{eq:stronger-containment-Udual-tensor-Udual}.
  Let us consider the filtration
  \begin{equation*}
    V^{\omega_5} \otimes \cU^\vee = [\cU^{\vee} \otimes \cU^\vee (-1), \,
    \wedge^3\cU^{\vee} \otimes \cU^\vee (-1), \, \cU \otimes \cU^\vee (1)]
  \end{equation*}
  provided by Lemma \ref{lemma:filtrations-spinor-representations}.
  Twisting it by $\cO(t)$ with $t \in [1,2]$ and using
  \eqref{eq:stronger-containment-Udual-tensor-Lambda^3Udual},
  \eqref{eq:containment-U-tensor-Udual},
  \eqref{eq:containment-Udual}
  we obtain $\cU^{\vee} \otimes \cU^\vee(t) \in \cD$ for $t \in [0,1]$.
  Combining this with \eqref{eq:containment-Udual-tensor-Udual} we get the claim.
\end{proof}

\begin{lemma}\label{lemma:auxiliary-2}
  For any $j \in [0,5]$ we have
  \begin{equation}\label{eq:containment-for-fullness}
    \cU^\vee \otimes \wedge^j \cE^\vee(t) \in \cD \quad \text{for} \quad t \in [0,7].
  \end{equation}
\end{lemma}

\begin{proof}
  Let us consider the exact sequence
  \begin{equation*}
    0 \to \cO_X \to \cU^\vee \to \cE^\vee \to 0.
  \end{equation*}
  It implies that $\wedge^j \cU^\vee$ has a filtration with factors $\cO_X, \, \cE^\vee, \, \wedge^2 \cE^\vee, \, \dots, \, \wedge^j \cE^\vee$. Therefore, arguing inductively with respect to $j$, if we know
  the inclusions
  \begin{equation}\label{eq:containment-for-fullness-prelim}
    \cU^\vee \otimes \wedge^j \cU^\vee(t) \in \cD \quad \text{for} \quad t \in [0,7],
  \end{equation}
  for all $j \in [0,5]$, then we know \eqref{eq:containment-for-fullness}.
  Now we note that \eqref{eq:containment-for-fullness-prelim} holds by
  Corollary \ref{corollary:almost-all-containments},
  \eqref{eq:stronger-containment-Udual-tensor-Lambda^3Udual} and
  \eqref{eq:stronger-containment-Udual-tensor-Udual}.
\end{proof}

\section{Proof of fullness}

\label{section:fullness}

We are now in position to prove fullness of our Lefschetz exceptional collection. Recall that this is defined in \eqref{eq:collection-d6-p6}.

\medskip

\begin{theorem}
The semiorthogonal exceptional collection appearing in \eqref{eq:collection-d6-p6} is full.
\end{theorem}

\begin{proof}
  Let us take an object $F \in \cD^\perp$, i.e. we have
  \begin{equation*}
    \Ext_X^\bullet(A,F) = 0 \quad \text{for any} \quad A \in \cD.
  \end{equation*}
  Let $s \in H^0(X,\cE^\vee)$ be a general section and $\is \colon Y_s \to X$ the
  embedding of its zero locus, as in Lemma \ref{lemma:auxiliary-1}(1).

  Let us consider the set of vector bundles on X defined by
  \begin{equation*}
    \Sigma \coloneqq \{ \cU^\vee(t) \, \vert \, t \in [0,7] \}.
  \end{equation*}
  By Lemma \ref{lemma:auxiliary-2} for any $E \in \Sigma$ and any $j$ the bundle
  $E \otimes \wedge^j \cE^\vee$ lies in $\cD$. Hence, we have
  \begin{equation*}
    \Ext_X^\bullet(E \otimes \wedge^j \cE^\vee, F) =
    H^\bullet(X, \wedge^j \cE \otimes E^\vee \otimes F) = 0 \quad \text{for all } j,
  \end{equation*}
  and making use of the Koszul complex
  \begin{equation*}
    0 \to \wedge^5 \cE \to \dots \to \cE \to \cO_X \to \is_* \cO_{Y_s} \to 0,
  \end{equation*}
  we obtain
  \begin{equation*}
    H^\bullet(X, \left( E^\vee \otimes F \right) \otimes \is_* \cO_Y) = 0.
  \end{equation*}
  Now, by projection formula we rewrite
  \begin{equation*}
    H^\bullet(X, \left( E^\vee \otimes F \right) \otimes \is_* \cO_Y) = H^\bullet(Y_s, \is^* \left( E^\vee \otimes F \right))
    = \Ext^\bullet_Y(\is^* E , \is^* F) = 0.
  \end{equation*}

  Recall that $Y_s \simeq \OG(5,10)$ has two connected components $Y_{s+}$ and $Y_{s-}$.
  We denote the compositions $Y_{s\pm} \subset Y_s \overset{\is}{\to} X$ by $\is_{\pm}$.
  Using this notation we have
  \begin{equation*}
    \Ext^\bullet_{Y_s}(\is^* E , \is^* F) =
    \Ext^\bullet_{Y_{s+}}(\is_+^* E , \is_+^* F) \oplus \Ext^\bullet_{Y_{s-}}(\is_-^* E , \is_-^* F).
  \end{equation*}
  Hence, we have
  \begin{equation*}
    \Ext^\bullet_{Y_{s+}}(\is_+^* E , \is_+^* F) = 0
    \quad \text{and} \quad
    \Ext^\bullet_{Y_{s-}}(\is_-^* E , \is_-^* F) = 0.
  \end{equation*}
  Applying Lemma \ref{lemma:auxiliary-1}(2) and Theorem \ref{theorem:collection-og-5-10}
  we obtain $\is_+^* F = 0 $ and $ \is_-^* F = 0$. Hence, we conclude $\is^* F = 0$. Finally, since the above argument works for any general $s \in H^0(X,\cE^\vee)$,
  by Lemma \ref{lemma:auxiliary-1}(3,4) we obtain $F = 0$.
\end{proof}

\section{A collection on the Freudenthal variety}
\label{section:E7}

Recall from the introduction that, in the third row of Freudenthal magic square, the homogeneous varieties $\bP^2\times \bP^2$, $G(3,6)$, $X=\Spin_{12}/\rP_6$ and $\rE_7/\rP_7$ appear, so by the general philosophy of \cite{LANMAN} these varieties should share a similar geometric behaviour. 
Our first observation here is that Lemma \ref{lemma:canonical-extensions} can be transposed to $\rE_7/\rP_7$. 

\begin{lemma}
On $\rE_7/\rP_7$ we have a canonical $\rE_7$-equivariant exceptional extension
$$ 0\to \cO_{\rE_7/\rP_7} \to O \to \cU_{\omega_1}\to 0, $$
with $O^\vee(2)$ being the normal bundle of $\rE_7/\rP_7$ inside $\bP(V^{\omega_7})$.
\end{lemma}

\begin{proof}
In the proof of Lemma \ref{lemma:canonical-extensions} substitute: $\Spin_{12}$ with $\rE_7$, $\rP_6$ with $\rP_7$, $\omega_2$ with $\omega_1$, $\omega_6$ with $\omega_7$ and $P$ with $O$; the modified proof still holds.
\end{proof}

\begin{lemma}
The collection $\langle \cO_{\rE_7/\rP_7},O,\cO_{\rE_7/\rP_7}(1),O(1),\dots,\cO_{\rE_7/\rP_7}(17),O(17)\rangle$ is exceptional.
\end{lemma}

\begin{proof}
  By an application of the BBW Theorem, we get, for $1\leq i\leq 17$:
  $$\Ext_X^\bullet(\cO_{\rE_7/\rP_7}(i),\cO_{\rE_7/\rP_7})=\Ext_X^\bullet(O(i),\cO_{\rE_7/\rP_7})=\Ext_X^\bullet(O(i),O)=0.$$
  Since $O$ is a non-trivial extension of $\cO_{\rE_7/\rP_7}$ and $\cU_{\omega_1}$, we also get $\Ext_X^\bullet(O,\cO_{\rE_7/\rP_7})=0$.
\end{proof}

One can also define a $\rG$-equivariant extension 
$$
0 \to O\to P' \to \cU_{2\omega_1} \to 0
$$
Let us define a \emph{numerical exceptional collection} in the derived category $\Db(X)$ of any smooth projective variety $X$ as a collection of objects $E_1,\dots,E_r$ such that $\chi(E_i,E_j)=0$ if $i>j$ and $\chi(E_i,E_i)=1$ for all $i$. Let us denote by
$$ 
\cB':=\left( \cO_{\rE_7/\rP_7},O,P' \right).
$$

Moreover we will denote by $Q'$ the projection of $\cU_{\omega_1+\omega_3}(-5)$ to the left orthogonal of
$$\langle \cB',\dots,\cB'(17)\rangle.$$
\begin{remark}
\label{rem_beware}
Here and later on by ``projection'' we mean that $Q'$ is obtained as an extension of $\cU_{\omega_1+\omega_3}(-5)$ with elements in the collection $( \cB',\dots,\cB'(17))$ so that, for any element $E\in ( \cB',\dots,\cB'(17))$, $\chi(E,Q'')=0$; if we knew that $( \cB',\dots,\cB'(17))$ were an exceptional collection, then it would be admissible and the ``projection'' to its left orthogonal would be well defined. Notice however that $Q'$ is uniquely defined in the Grothendieck group.
\end{remark}

We consider the collection:
$$ \cA':= \left( Q',\cO_{\rE_7/\rP_7},O,P' \right) .$$

\begin{proposition}
\label{first_exc_coll_E7}
The collection $\left( \cA',\cA'(1),\cB'(2),\dots,\cB'(17)\right)$ is a numerical exceptional collection of maximal length, i.e. of length equal to $\sum_p h^{p,p}(\rE_7/\rP_7)=56$.
\end{proposition} 

\begin{proof}
The projection $Q'$ can be computed numerically, i.e. in the Grothendieck group of $\rE_7/\rP_7$. Let us explain the strategy. Let us denote by $R_0:=\cU_{\omega_1+\omega_3}(-5)$ and let us define $R_1,R_2,\dots,R_{54}=Q'$ inductively. Write the collection $\left( \cB',\dots,\cB'(17)\right)$ as $\left( E_1,\dots,E_{54}\right)$. The object $R_{i+1}$ will be an extension of $R_i$ by $\chi(E_{i+1},R_i) E_{i+1}(-18)$ in the Grothendieck group. When this process finishes, by Serre duality one obtains an object $Q'$ which is by definition left orthogonal to $\langle E_1,\dots,E_{54}\rangle$, and one checks that $\chi(Q',Q')=-\chi(Q'(2),Q')=1$ and $\chi(Q'(1),Q')=0$. Another computation with BBW Theorem yields the numerical exceptionality of the collection.
\end{proof}

Some observations are in order. Let us write the element in the Grothendieck group corresponding to $Q'$:
$$ \cU_{\omega_1+\omega_3}(-5) -P'(-7)+O(-6)+56P'(-6)-1673\cO_{\rE_7/\rP_7}(-5) -3137O(-5)+ P'(-5) +$$ $$-94656
\cO_{\rE_7/\rP_7}(-4)-56  P'(-4) -54342
\cO_{\rE_7/\rP_7}(-3) +3271 O(-3)- P'(-3) -58576
\cO_{\rE_7/\rP_7}(-2) +$$ $$-968 O(-2)+56 P'(-2)
+54342\cO_{\rE_7/\rP_7}(-1) -3137 O(-1).$$
Notice that by general properties of mutations we also obtain another numerical exceptional collection:
$$ \left( Q', \rL_{\cB}Q'(1),\cB',\cB'(1),\dots,\cB'(17) \right) .$$

The peculiar fact about this collection is that its residual collection $\left( Q', \rL_{\cB'}Q'(1) \right) $ is numerically completely orthogonal, meaning that $\chi(Q',\rL_{\cB'}Q'(1))=\chi(\rL_{\cB'}Q'(1),Q')=0$. 
Therefore, such a residual collection numerically satisfies Dubrovin's refined conjecture, see \cite[Conjecture 1.3]{KS21} and \cite[Corollary 1.2]{ChMaPe3}. We believe that the collection above is an exceptional collection in $\Db(\rE_7/\rP_7)$ (of maximal length), but we could not prove our claim due to the big number of cohomologies between $\cHom$'s of the irreducible factors of the extensions in play. We even suspect that the collection is full. 

If the above collection has the advantage of respecting Dubrovin's conjecture's expectation, we will briefly describe another numerical exceptional collection on $\rE_7/\rP_7$ which is closer to the collection of \eqref{eq:collection-d6-p6} on $\Spin_{12}/\rP_6$. Let us begin with the usual collection $\left( \cO_{\rE_7/\rP_7},O,\dots, \cO_{\rE_7/\rP_7}(17),O(17) \right)$. Consider the projection $P$ of $\cU_{\omega_3}$ to the left orthogonal of $\langle \cO_{\rE_7/\rP_7}(1),O(1),\dots, \cO_{\rE_7/\rP_7}(18),O(18) \rangle$. Moreover consider the projection (as in Remark \ref{rem_beware}) $Q$ of $\cU_{\omega_1+\omega_3}$ to the left orthogonal of $\langle \cO_{\rE_7/\rP_7}(1),O(1),P(1) ,\dots,\cO_{\rE_7/\rP_7}(18),O(18),P (18) \rangle$. Let us write
\[\cB:=\langle \cO_{\rE_7/\rP_7},O,P \rangle, \qquad \cA:=\langle \cO_{\rE_7/\rP_7},O,P,Q \rangle.\]

By a repeated application of BBW Theorem done with a Python script using \cite{lie} as in the proof of Proposition \ref{first_exc_coll_E7} one obtains the following result. 
\begin{proposition}
\label{prop_E7_second_coll}
The homogeneous bundles $P$ and $Q$ are numerically exceptional and the collection
  $$ \cD':= \left( \cA,\cA(1),\cB(2),\dots,\cB(17) \right) $$ is numerically exceptional of maximal length. Moreover $Q$ and $\rL_{\cB(1)}Q(1)$ are numerically completely orthogonal.
\end{proposition}
The computation is enclosed as an ancillary file.

\begin{remark}
Notice that in the Python script, in order to obtain the result, it was easier to work with the projection (as in Remark \ref{rem_beware}) of $\cU_{\omega_1+\omega_3}(-9)$ to the left orthogonal of $\langle \cB,\cB(1),\dots,\cB(17) \rangle$. Then $Q$ is easily obtained as the projection of $F(9)$ to the left of $\langle \cB(1),\dots,\cB(8) \rangle$. Similarly, if $F'$ is the projection of $\cU_{\omega_1+\omega_3}(-8)$ to the left of $\langle \cB,\cB(1),\dots,\cB(17) \rangle$, then $\rL_{\cB(1)}Q(1)$ is obtained as the projection of $F'(9)$ to the left of $\langle \cB(1),\dots,\cB(8)\rangle$. Since these operations preserve (numerical) orthogonality, we prove Proposition \ref{prop_E7_second_coll} using $F$ and $F'$.
\end{remark}

\bibliographystyle{amsalpha}
\bibliography{refs}

\end{document}